\documentclass[a4paper,11pt]{article}
\usepackage[french, english]{babel}
\usepackage[latin1]{inputenc}
\usepackage{amsfonts}
\usepackage{amsmath}
\usepackage{amsthm}
\usepackage{amssymb}
\usepackage{mathrsfs}
\usepackage{ae, aecompl}
\usepackage{graphicx}
\usepackage[a4paper, portrait]{geometry}
\usepackage{sectsty}
\usepackage{lmodern}
\RequirePackage{mathrsfs}
\let\mathcal\mathscr
\usepackage{url}

\newcommand*{\Z}{\ensuremath{\mathbf{Z}}}                        
\newcommand*{\qp}{\ensuremath{\mathbf{Q}_p}}                     
\newcommand*{\zp}{\ensuremath{\mathbf{Z}_p}}

\newcommand*{\zpet}{\ensuremath{\mathbf{Z}_p^*}}
\newcommand*{\qpet}{\ensuremath{\mathbf{Q}_p^*}}
\DeclareMathOperator{\p1}{\bold{P}^1}

\newtheoremstyle{theorem}{11pt}{11pt}{\slshape}{}{\bfseries}{.}{.5em}{}
\newtheoremstyle{note}{11pt}{11pt}{}{}{\bfseries}{.}{.5em}{}

\theoremstyle{plain}
  \newtheorem{theorem}{Th\'{e}or\`{e}me}[section]
  \newtheorem{proposition}[theorem]{Proposition}
  \newtheorem{lemma}[theorem]{Lemme}
  \newtheorem{corollary}[theorem]{Corollaire}

\theoremstyle{definition}
  \newtheorem{definition}[theorem]{D\'{e}finition}

\theoremstyle{remark}
  \newtheorem{example}[theorem]{Exemple}
  \newtheorem{remark}[theorem]{Remarque}

\begin{document}

\selectlanguage{french}

\title{Actions infinitésimales dans la correspondance de Langlands
locale $p$-adique}

\author{Gabriel Dospinescu \thanks
  {\footnotesize {C.M.L.S, Ecole Polytechnique,
  gabriel.dospinescu@math.polytechnique.fr}}}

\maketitle

\begin{abstract}
\footnotesize

  Let $V$ be a two-dimensional absolutely irreducible
  $\overline{\qp}$-representation of ${\rm Gal}(\overline{\qp}/\qp)$
  and let $\Pi(V)$ be the ${\rm GL}_2(\qp)$ Banach representation
  associated by Colmez's $p$-adic Langlands correspondence.
  We establish a link between the action of the Lie algebra of ${\rm GL}_2(\qp)$
  on the locally analytic vectors $\Pi(V)^{\rm an}$ of $\Pi(V)$,
  the connection $\nabla$ on the $(\varphi,\Gamma)$-module
  associated to $V$ and the Sen polynomial of $V$.
  This answers a question of Harris, concerning
  the infinitesimal character of $\Pi(V)^{\rm an}$. Using this result,
  we give a new proof of a theorem of Colmez, stating that
$\Pi(V)$ has nonzero locally algebraic vectors if and only if $V$ is potentially semi-stable
  with distinct Hodge-Tate weights.

\end{abstract}

\selectlanguage{french}

\section{Introduction}
\label{intro}

 Cet article s'inscrit dans le cadre de la correspondance de Langlands locale $p$-adique,
 imagin\'{e}e par Breuil~\cite{Br} (au moins dans le cas potentiellement semi-stable) et \'{e}tablie par Colmez~\cite{Cbigone} pour ${\rm GL}_2(\qp)$,
\`{a} la suite des travaux de Colmez~\cite{Cserieunit}, Berger-Breuil~\cite{BB}, Kisin~\cite{KiAst}.
Cette correspondance encode la correspondance de Langlands
 locale classique pour ${\rm GL}_2(\qp)$.

 Soit $L$ une extension finie de $\qp$ et soit $V$ une $L$-repr\'{e}sentation absolument irr\'{e}ductible de ${\rm Gal}(\overline{\qp}/\qp)$,
 de dimension $2$. En utilisant la th\'{e}orie des $(\varphi,\Gamma)$-modules de Fontaine~\cite{FoGrot}, Colmez associe \`{a} $V$
une ${\rm GL}_2(\qp)$-repr\'{e}sentation de Banach $\Pi(V)$, unitaire, admissible, topologiquement irr\'{e}ductible et d\'{e}crit les vecteurs localement analytiques
 $\Pi(V)^{\rm an}$ de $\Pi(V)$ en fonction de $V$. Cependant, la description est tr\`{e}s indirecte, ce qui rend
l'\'etude de
 $\Pi(V)^{\rm an}$ un peu d\'elicate: par exemple, on ne sait pas si $\Pi(V)^{\rm an}$ est toujours de longueur finie.
Dans cet article, on calcule
 l'action infinit\'{e}simale de ${\rm GL}_2(\qp)$ sur $\Pi(V)^{\rm an}$, et on en d\'eduit
une nouvelle d\'emonstration d'un r\'esultat de Colmez caract\'erisant les repr\'esentations $V$
telles que $\Pi(V)$ poss\`ede des vecteurs localement alg\'ebriques (c'est ce r\'esultat qui
permet de faire le pont avec la correspondance classique).
Il y a d'autres probl\`{e}mes auxquels les techniques de cet article s'appliquent, le cas le plus notable \'{e}tant (voir \cite{Cvectan}
et \cite{Jacquet})
 des conjectures de Berger, Breuil~\cite{BB} et Emerton~\cite{EmCoates}, qui d\'{e}crivent de mani\`{e}re pr\'{e}cise l'espace $\Pi(V)^{\rm an}$
 quand $V$ est trianguline.

\subsection{Notations}
 On fixe dans toute la suite un nombre premier $p$, une extension finie $L$ de $\qp$ et une cl\^{o}ture alg\'{e}brique $\overline{\qp}$ de $\qp$, contenant $L$.
  Soit $\chi:{\rm Gal}(\overline{\qp}/\qp)\to \zpet$ le caract\`{e}re cyclotomique; alors $\chi$ induit
un isomorphisme de
$ \Gamma={\rm Gal}({\qp}(\mu_{p^{\infty}})/\qp)$ sur $\zpet$, et on note $a\mapsto\sigma_a$ l'isomorphisme
r\'eciproque de $\zpet$ sur $\Gamma$ (on a donc $\sigma_a(\zeta)=\zeta^a$ pour tout $\zeta\in\mu_{p^{\infty}}$).
Soient $\mathcal{R}$
   l'anneau de Robba (c'est l'anneau des s\'{e}ries de Laurent $\sum_{n\in\mathbf{Z}} a_nT^n$ \`{a} coefficients dans $L$, qui convergent sur une couronne du type
   $0<v_p(T)\leq r$, o\`{u} $r$ d\'{e}pend de la s\'{e}rie), ${\cal E}^\dagger$ le sous-anneau de ${\cal R}$ des \'el\'ements born\'es
(c'est un corps) et ${\cal E}$ le compl\'et\'e de ${\cal E}^\dagger$ pour la valuation $p$-adique.
On munit ces anneaux d'actions continues de $\Gamma$ et d'un Frobenius $\varphi$, commutant entre elles,
en posant $\varphi(T)=(1+T)^p-1$ et $\sigma_a(T)=(1+T)^a-1$ si $a\in\zpet$.

\subsection{Le faisceau $\Delta\to \Delta\boxtimes U$}\label{faisceau}

    Une $L$-repr\'{e}sentation de ${\rm Gal}(\overline{\qp}/\qp)$
   est un $L$-espace vectoriel $V$ de dimension finie, muni d'une action $L$-lin\'{e}aire
   continue de  ${\rm Gal}(\overline{\qp}/\qp)$. Soit $V$ une $L$-repr\'{e}sentation de dimension
   $2$ de ${\rm Gal}(\overline{\qp}/\qp)$, absolument irr\'{e}ductible.
La th\'{e}orie de Fontaine~\cite{FoGrot}, compl\'et\'ee par des travaux de
    Cherbonnier-Colmez~\cite{CCsurconv}, Berger~\cite{Ber} et Kedlaya~\cite{Ked}, associe \`{a} $V$ des modules
   $D=D(V)$, $D^\dagger$, $D_{\rm rig}$, libres de rang $2$ sur $\mathcal{E}$, ${\cal E}^\dagger$ et ${\cal R}$,
munis d'une action semi-lin\'{e}aire de $\Gamma$ et d'un Frobenius $\varphi$ semi-lin\'{e}aire, commutant \`{a} $\Gamma$
(le module $D^\dagger$ n'a pas d'int\'er\^et propre, mais il permet de passer de $D$ \`a $D_{\rm rig}$
et vice-versa).  Enfin, on dispose sur
$\Delta\in\{D,D_{\rm rig}\}$ d'un
inverse \`{a} gauche $\psi$ de $\varphi$ qui commute \`a $\Gamma$.

   On peut utiliser ces op\'{e}rateurs
    pour d\'{e}finir une action\footnote{Voir la remarque \ref{explic} pour l'origine et la signification de ces formules.}
 du mono\"{\i}de $P^+=\left(\begin{smallmatrix} \zp- \{0\} & \zp \\0 & 1\end{smallmatrix}\right)$ sur $\Delta$, en posant
 $$\left(\begin{smallmatrix} p^ka & b \\0 & 1\end{smallmatrix}\right)z=(1+T)^b\cdot \varphi^k(\sigma_a(z))$$ pour $z\in \Delta$ et $k\geq 0$,
$a\in\zpet$, $b\in\zp$.
     Le mono\"{\i}de $P^+$ agit
   aussi sur $\zp$, par $\left(\begin{smallmatrix} a & b \\0 & 1\end{smallmatrix}\right)x=ax+b$. Colmez d\'{e}finit un
    faisceau $U\to \Delta\boxtimes U$ sur $\zp$,
   \'{e}quivariant sous l'action de $P^+$, dont les sections sur $a+p^n\zp$ sont
   $$\Delta \boxtimes (a+p^n\zp)=(1+T)^a\cdot \varphi^n(\Delta)\subset \Delta$$ (en particulier
   $\Delta\boxtimes \zp=\Delta$) et la restriction ${\rm{Res}}_{a+p^n\zp}: \Delta=\Delta\boxtimes \zp\to \Delta\boxtimes (a+p^n\zp)$ est
   $${\rm{Res}}_{a+p^n\zp}=
   \left(\begin{smallmatrix} 1 & a \\0 & 1\end{smallmatrix}\right)\circ \varphi^n
   \circ \psi^n\circ \left(\begin{smallmatrix} 1 & -a \\0 & 1\end{smallmatrix}\right).$$

 Faisons agir ${\rm GL}_2(\qp)$ sur ${\p1}(\qp)$ par $\left(\begin{smallmatrix} a & b \\c & d\end{smallmatrix}\right)z=
 \frac{az+b}{cz+d}$ (la restriction de cette action \`{a} $P^+$ est celle d\'{e}finie ci-dessus). Un des r\'{e}sultats principaux de~\cite{Cbigone} est le fait que
 le faisceau pr\'{e}c\'{e}dent sur $\zp$ s'\'{e}tend en un faisceau sur ${\p1}(\qp)$,
 \'{e}quivariant sous l'action de ${\rm GL}_2(\qp)$ (les formules pour l'action de ${\rm GL}_2(\qp)$ sont nettement
plus compliqu\'ees que celles ci-dessus, mais sont en fait inutiles pour les applications).  Soit
$w=\left(\begin{smallmatrix} 0 & 1 \\1 & 0 \end{smallmatrix}\right)$.
Comme $\p1(\qp)$ est obtenu par recollement \`{a} partir de $\zp$ et
  $w\cdot \zp$ le long de $\zpet$, on peut d\'{e}crire l'espace des sections globales du faisceau
  attach\'{e} \`{a} $\Delta$ par
  $$\Delta\boxtimes\p1=\{(z_1,z_2)\in \Delta\times \Delta| \quad {\rm Res}_{\zpet}(z_2)=
w_D({\rm Res}_{\zpet}(z_1))\},$$
o\`u $w_D$ est la restriction de $w$ \`a $\Delta\boxtimes\zpet$. Soit $\delta_D=\chi^{-1}\cdot \det V$, que l'on voit
comme caract\`{e}re de $\qpet$ par la th\'{e}orie locale du corps de classes\footnote{Donc $\chi$, vu comme caract\`{e}re de $\qpet$ est
$x\to x\cdot |x|_p$, o\`{u} $|\cdot|_p$ est la valeur absolue $p$-adique.}.
La repr\'{e}sentation $\Pi(V)$ attach\'{e}e \`{a} $V$ par la correspondance de Langlands locale $p$-adique
vit alors dans une suite exacte $$0\to \Pi(V)^*\otimes (\delta_D\circ \det)\to D\boxtimes\p1\to \Pi(V)\to 0$$
de ${\rm GL}_2(\qp)$-modules topologiques. Ses vecteurs localement analytiques
vivent dans une suite exacte analogue
      $$0\to (\Pi(V)^{\rm an})^*\otimes (\delta_D\circ \det)\to D_{\rm rig}\boxtimes\p1\to \Pi(V)^{\rm an}\to 0.$$
Par ailleurs, si $U$ est un ouvert compact de $\p1(\qp)$, on dispose d'une application de prolongement par $0$ de $\Delta\boxtimes U$ dans
$\Delta\boxtimes\p1$, ce qui permet en particulier de voir $D_{\rm rig}=D_{\rm rig}\boxtimes \zp$
comme un sous-module de
     $D_{\rm rig}\boxtimes \p1$.

\subsection{Actions infinit\'{e}simales}

  Dans~\cite{Cbigone}, Colmez montre que $D_{\rm rig}\boxtimes\p1$ est un module sous l'action\footnote{Plus pr\'{e}cis\'{e}ment,
  il d\'{e}montre que $D_{\rm rig}\boxtimes \p1$ est un module sous l'action de l'alg\`{e}bre des distributions
  d'un sous-groupe ouvert compact suffisament petit de ${\rm GL}_2(\qp)$.}
     de l'alg\`{e}bre enveloppante $U(\mathfrak{gl_2})$ de ${\rm Lie}({\rm GL}_2(\qp))$, et cette action
stabilise $D_{\rm rig}\boxtimes U$ pour tout ouvert compact $U$ de $\p1(\qp)$;
en particulier, cette action stabilise
     le sous-module $D_{\rm rig}=D_{\rm rig}\boxtimes \zp$. Soit
     $$h=\left(\begin{smallmatrix} 1 & 0 \\0 & -1\end{smallmatrix}\right), \quad
   u^+=\left(\begin{smallmatrix} 0 & 1 \\0 & 0\end{smallmatrix}\right),
\quad u^-=\left(\begin{smallmatrix} 0 & 0 \\1 & 0\end{smallmatrix}\right)$$
     la base usuelle de $\mathfrak{sl_2}$, de sorte que ${\rm I}_2=\big(\begin{smallmatrix}1&0\\0&1\end{smallmatrix}\big)$, $u^+$, $u^-$,
   $h$ forment une base de $\mathfrak{gl_2}$. D'apr\`{e}s Berger~\cite{Ber}, $D_{\rm rig}$
     est aussi muni d'une action infinit\'{e}simale de $\Gamma$, \`{a} travers la connexion
   $$\nabla(z)=\lim_{\gamma\to 1}\frac{\gamma(z)-z}{\chi(\gamma)-1}.$$
  Enfin, rappelons que $t=\log(1+T)\in\mathcal{R}$ et que, gr\^{a}ce aux travaux de Sen~\cite{Sen},
   on peut associer \`{a} $V$ un polyn\^{o}me
   $P_{\rm Sen,V}\in L[X]$ de degr\'{e} $2$, unitaire, dont les racines s'appellent
   les poids de Hodge-Tate g\'{e}n\'{e}ralis\'{e}s de $V$. Si $V$ est Hodge-Tate, il s'agit des poids
   de Hodge-Tate classiques de $V$ (avec la convention que le poids de ${\qp}(1)$ est $1$).

   \begin{theorem}
\label{main1}
     Soit $V$ une $L$-repr\'{e}sentation absolument irr\'{e}ductible
     de ${\rm Gal}(\overline{\qp}/\qp)$, de dimension $2$, \`{a} poids de Hodge-Tate $a,b$, de telle sorte
     que $P_{\rm Sen, V}=(X-a)(X-b)$.
     L'action de $\mathfrak{gl_2}$ sur $D_{\rm rig}$ est donn\'{e}e par ${\rm I}_2 (z)=(a+b-1)z$,
     $$\quad u^-(z)=-\frac{P_{\rm Sen,V}(\nabla)(z)}{t}, \quad h(z)=2\nabla(z)-(a+b-1)z, \quad u^+(z)=tz.$$
   \end{theorem}

    On d\'{e}duit du th\'{e}or\`{e}me~\ref{main1} l'action de $U(\mathfrak{gl_2})$ sur
    $D_{\rm rig}\boxtimes \p1$, car un \'{e}l\'{e}ment de $D_{\rm rig}\boxtimes \p1$
    est de la forme $z_1+w\cdot z_2$ avec $z_1,z_2\in D_{\rm rig}$. La preuve du th\'{e}or\`{e}me \ref{main1} fait jouer un
r\^ole essentiel \`a l'\'{e}l\'{e}ment de Casimir $$C=u^+u^-+u^-u^++\frac{1}{2}h^2\in U(\mathfrak{sl_2}),$$ qui engendre le centre
de l'alg\`{e}bre ${\rm U}(\mathfrak{sl_2})$. En effet, il n'est pas difficile de v\'{e}rifier que $C$ induit
    un endomorphisme du $(\varphi,\Gamma)$-module $D_{\rm rig}$. Des propri\'{e}t\'{e}s
    standard des $(\varphi,\Gamma)$-modules entra\^{i}nent que cet endomorphisme est scalaire. Pour identifier ce scalaire,
on utilise la th\'{e}orie de Hodge $p$-adique, en particulier les techniques diff\'{e}rentielles de Berger et
Fontaine (voir la partie \ref{FontaineSen}). Le th\'{e}or\`{e}me~\ref{main1} est une cons\'{e}quence facile du r\'{e}sultat suivant, qui
r\'{e}pond aussi \`{a} une question de Harris~(\cite{EmCoates}, remark 3.3.8).

 \begin{theorem}
 \label{infchar}
  Si $V$ est comme dans le th\'{e}or\`{e}me \ref{main1}, l'action de $C$ sur
  $D_{\rm rig}\boxtimes \p1$, et donc sur $\Pi(V)^{\rm an}$,
  est la multiplication par $\frac{(b-a)^2-1}{2}$.
 \end{theorem}

    En fait, les deux th\'{e}or\`{e}mes pr\'{e}c\'{e}dents sont \'{e}quivalents \`{a} l'\'{e}nonc\'{e} suivant, portant
sur l'involution $w_D$ de Colmez.

    \begin{corollary}
    \label{conseq1}
      Si $z\in D_{\rm rig}\boxtimes\zpet$, on a $$w_D(tz)=-\frac{P_{\rm Sen,V}(\nabla)(w_D(z))}{t}.$$
    \end{corollary}

Si $a,b\in L$, l'espace des repr\'esentations $V$ dont les poids de Hodge-Tate
g\'en\'eralis\'es sont $a,b$ est une vari\'et\'e de dimension~$3$.
   Les repr\'{e}sentations de Banach attach\'{e}es \`{a} ces repr\'{e}sentations
   galoisiennes ont le m\^{e}me caract\`{e}re infinit\'{e}simal d'apr\`{e}s le th\'{e}or\`{e}me \ref{infchar}, mais elles sont deux \`{a} deux
   non isomorphes (la correspondance de Langlands locale
    $p$-adique \'{e}tant injective, d'apr\`{e}s un th\'{e}or\`{e}me de Colmez). On voit ainsi qu'il y a une infinit\'{e} de repr\'{e}sentations
   de Banach de ${\rm GL}_2(\qp)$, absolument irr\'{e}ductibles, unitaires, admissibles,
   ayant le m\^{e}me caract\`{e}re infinit\'{e}simal. Cela ne se produit
   pas dans la th\'{e}orie classique des repr\'{e}sentations unitaires
   des groupes r\'{e}els (semi-simples, mais oublions cela
   pour un moment), d'apr\`{e}s des r\'{e}sultats classiques (mais profonds) de Harish-Chandra.

    Le th\'{e}or\`{e}me \ref{infchar} coupl\'{e} \`{a} un th\'{e}or\`{e}me r\'{e}cent et tr\`{e}s d\'{e}licat
    de Pa\v{s}k\={u}nas~\cite{Pa}, qui d\'{e}montre la surjectivit\'{e} de la
    correspondance $V\to \Pi(V)$, donne le r\'{e}sultat suivant:

 \begin{theorem}
 \label{conseq2}
  Soit $p>3$ et soit $\Pi$ une $L$-repr\'{e}sentation
  de Banach de ${\rm GL}_2(\qp)$, unitaire, admissible, absolument irr\'{e}ductible.
Alors $\Pi^{\rm an}$ admet un caract\`{e}re infinit\'{e}simal.
 \end{theorem}

  L'hypoth\`{e}se $p>3$ appara\^{i}t dans les travaux de Pa\v{s}k\={u}nas; il est raisonnable de penser que le r\'{e}sultat
reste vrai pour $p\leq 3$
(cf.~\cite{DfoncteurColmez} pour des r\'esultats dans cette direction).
 Pa\v{s}k\={u}nas d\'{e}montre\footnote{Sa preuve est tr\`{e}s d\'{e}tourn\'{e}e, mais une preuve plus simple d'un r\'{e}sultat
plus g\'{e}n\'{e}ral est donn\'{e}e dans~\cite{DBenjamin}. Elle repose sur une version du lemme de Quillen due \`{a}
Ardakov et Wadsley~\cite{AW}.} en outre l'analogue du lemme de Schur: si $\Pi$ est comme dans
   le th\'{e}or\`{e}me \ref{conseq2}, alors
   ${\rm End}_{L[\rm GL_2(\qp)]}^{\rm cont}(\Pi)=L$. On peut se demander si le m\^{e}me r\'{e}sultat
   est valable en rempla\c{c}ant $\Pi$ par $\Pi^{\rm an}$. Cela d\'{e}montrerait le th\'{e}or\`{e}me
   \ref{conseq2} de mani\`{e}re plus directe (mais pas
   le th\'{e}or\`{e}me \ref{infchar}), mais je ne sais pas le faire.
   Je ne sais pas non plus si $\Pi^{\rm an}$ est de longueur finie pour une
   repr\'{e}sentation $\Pi$ comme dans le th\'{e}or\`{e}me \ref{conseq2}, m\^{e}me s'il est tr\`{e}s probable que
   la r\'{e}ponse soit positive\footnote{ Voir aussi la conjecture faite dans~\cite{DBenjamin} et ce que ce l'on sait
pour l'instant concernant cette question. Le cas o\`{u} $V$ est trianguline est d\'{e}montr\'{e} dans~\cite{Cvectan}, o\`{u}
   $\Pi(V)^{\rm an}$ est \'{e}tudi\'{e} en d\'{e}tail, r\'{e}pondant ainsi \`{a} des conjectures de Berger, Breuil et Emerton.
 L'ingr\'{e}dient principal de la preuve est l'\'{e}tude du module de Jacquet analytique de $\Pi(V)^{\rm an}$, qui peut se
faire facilement \cite{Jacquet} en utilisant les r\'{e}sultats de cet article. }. Notons que dans la th\'{e}orie des groupes r\'{e}els,
   ce r\'{e}sultat est une cons\'{e}quence formelle du fait qu'il y a seulement un nombre fini de (classes de) repr\'{e}sentations
 unitaires irr\'{e}ductibles ayant un caract\`{e}re infinit\'{e}simal donn\'{e}, mais comme on l'a remarqu\'{e} ces r\'{e}sultats
 ne sont plus valables en $p$-adique.

     La motivation principale du th\'{e}or\`{e}me \ref{main1} \'{e}tait de donner une preuve
     "analytique" du th\'{e}or\`{e}me suivant, d\^{u} \`{a} Colmez~(\cite{Cbigone}, thm. VI.6.13, VI.6.18).

   \begin{theorem}
   \label{main2}
     Si $V$ est comme dans le th\'{e}or\`{e}me \ref{main1}, alors $\Pi(V)$ a des vecteurs
     localement alg\'{e}briques non nuls si et seulement si $V$ est potentiellement
     semi-stable \`{a} poids de Hodge-Tate distincts.
   \end{theorem}

   Notons que le th\'{e}or\`{e}me \ref{main2} joue un r\^{o}le crucial dans la preuve
   de la conjecture de Fontaine-Mazur en dimension $2$ (sous
   certaines hypoth\`{e}ses faibles) par Emerton~\cite{Emcomp}.
    Il joue aussi un r\^{o}le important dans la preuve par
      Kisin~\cite{KiFM} de la m\^{e}me conjecture (mais Kisin a besoin d'informations plus fines
      concernant les vecteurs localement alg\'{e}briques).

Notre d\'emonstration du
th\'{e}or\`{e}me \ref{main2} est nettement plus directe que celle que
donne Colmez dans~\cite{Cbigone}, qui repose
sur deux g\'{e}n\'{e}ralisations des lois de r\'{e}ciprocit\'{e} explicites de Kato~\cite{Kato} et Perrin-Riou~\cite{Perrin},~\cite{Cannals}.
     Il est assez amusant de constater que m\^{e}me si on utilise beaucoup
      des r\'{e}sultats de Colmez, ce sont pr\'{e}cis\'{e}ment ces r\'{e}sultats qui
      ne sont pas utilis\'{e}s dans sa preuve du th\'{e}or\`{e}me \ref{main2}.

\subsection{Remerciements} Ce travail est une partie de ma th\`{e}se de doctorat, r\'{e}alis\'{e}e
 sous la direction de Pierre Colmez et de Ga\"{e}tan Chenevier.
 Je leur suis profond\'{e}ment r\'{e}connaissant pour tout ce qu'ils m'ont appris.
 Il sera plus qu'\'{e}vident au lecteur combien cet article
 doit au travail monumental~\cite{Cbigone} de Colmez.
 Mais plus encore, je lui suis profond\'{e}ment reconnaissant
 pour les nombreuses et longues discussions que nous avons eues
 autour de cet article. En particulier, sans son observation
 que l'action de $\mathfrak{gl_2}$ se prolonge au module de Fontaine, cet article
 n'aurais jamais vu le jour. Je remercie vivement
 Wang Shanwen pour des discussions \'{e}clairantes autour
 de certains points de l'article, ainsi que le rapporteur, dont les remarques m'ont beaucoup aid\'{e}
 \`{a} am\'{e}liorer la pr\'{e}sentation.

\section{Th\'{e}orie de Hodge $p$-adique et $(\varphi,\Gamma)$-modules}

  Le but de ce chapitre est de rappeler quelques r\'{e}sultats standard concernant les liens entre la th\'{e}orie
  des $(\varphi,\Gamma)$-modules et la th\'{e}orie de Hodge $p$-adique. Ces r\'{e}sultats
  seront pleinement utilis\'{e}s dans les chapitres \ref{Kirillovetdualite} et \ref{deRham}, mais la proposition \ref{HC} sera
suffisante pour la d\'{e}monstration du th\'{e}or\`{e}me \ref{main1}. Pour expliquer certaines constructions classiques, il
 faut malheureusement ouvrir la bo\^{i}te de Pandore des anneaux de Fontaine.
  Le lecteur pourra facilement faire abstraction de la plupart de ces anneaux, car pour la suite seuls
 les anneaux $\mathbf{B}_{\rm dR}$, $\mathbf{\tilde{B}}$, $\mathbf{\tilde{B}}^+$ et $\mathcal{R}$ (voir plus bas pour les d\'{e}finitions)
 seront utilis\'{e}s. Voir~\cite{FoAnnals} et~\cite{Espvect} pour les preuves des assertions concernant les anneaux de Fontaine,
ainsi que~\cite{Monodromie} pour une vue d'ensemble.
  Soit $\varepsilon^{(n)}$ une racine primitive d'ordre $p^n$ de l'unit\'{e}, telle que
 $(\varepsilon^{(n+1)})^p=\varepsilon^{(n)}$ pour tout $n$.
 On pose $F_n=\qp(\varepsilon^{(n)})$, $L_n=L\otimes_{\qp} F_n$ et $L_{\infty}=\cup_{n} L_n$.

  \subsection{Anneaux de fonctions analytiques}\label{Fctan}

  Notons, pour $n\geq 1$, $$r_n=\frac{1}{p^{n-1}(p-1)}=v_p(\varepsilon^{(n)}-1).$$

  \begin{definition}

a) Soit $\mathcal{E}$ l'anneau des s\'{e}ries
   de Laurent $\sum_{k\in\mathbf{Z}}a_kT^k$, avec $(a_k)_{k\in\Z}$ une
   suite born\'{e}e d'\'{e}l\'{e}ments de $L$ et $\lim_{k\to-\infty} a_{k}=0$.

b) Soit $\mathcal{E}^{(0,r_n]}$ l'anneau des s\'{e}ries $f=\sum_{k\in\mathbf{Z}}
a_kT^k\in\mathcal{E}$ qui convergent sur la couronne $0<v_p(T)\leq r_n$.
Soit $\mathcal{E}^{\dagger}$ la r\'{e}union des $\mathcal{E}^{(0,r_n]}$.

c) Soit $\mathcal{E}^{]0,r_n]}$ l'anneau des s\'{e}ries $\sum_{k\in\mathbf{Z}} a_kT^k$
   qui convergent sur la couronne
   $0<v_p(T)\leq r_n$ et soit
   $\mathcal{R}$ l'anneau de Robba, r\'{e}union des $\mathcal{E}^{]0,r_n]}$
   \`{a} l'int\'{e}rieur de $L[[T,T^{-1}]]$.
  \end{definition}

  Tous ces anneaux sont munis de topologies naturelles\footnote{ $\mathcal{E}^{]0,r_n]}$ est muni de la topologie de Fr\'{e}chet d\'{e}duite
 des normes sup sur les couronnes $r_m\leq v_p(T)\leq r_n$, avec $m\geq n$ et
  $\mathcal{R}=\cup_{n}\mathcal{E}^{]0,r_n]}$ est muni de la topologie limite inductive. Soit $O_{\mathcal{E}}$ le sous-anneau
 de $\mathcal{E}$ form\'{e} des s\'{e}ries
  $\sum_{n\in\mathbf{Z}} a_nT^n$ avec $a_n\in O_L$. On munit $O_{\mathcal{E}}$
  de la topologie dont une base de voisinages de
  $0$ est donn\'{e}e par les $T^nO_L[[T]]+p^kO_{\mathcal{E}}$ et on met la topologie
  limite inductive sur $\mathcal{E}=\cup_{n} p^{-n}O_{\mathcal{E}}$.}, ainsi que
  d'une action
  continue de $\Gamma$, d\'{e}finie par $\sigma_a (f)(T)=f((1+T)^a-1)$ pour $a\in\zpet$. Les anneaux
  $\mathcal{E}$ et $\mathcal{R}$ sont aussi munis d'un Frobenius
  continu $\varphi$, d\'{e}fini par $\varphi(f)(T)=f((1+T)^p-1)$ et commutant
  \`{a} l'action de $\Gamma$. Si $\Lambda\in \{\mathcal{E}, \mathcal{R}\}$, tout \'{e}l\'{e}ment
  $f\in \Lambda$ s'\'{e}crit de mani\`{e}re unique sous la forme
  $$f=\sum_{i=0}^{p-1}(1+T)^i\varphi(f_i),$$
  avec $f_i\in \Lambda$. En posant $\psi(f)=f_0$, on
  obtient un endomorphisme continu
  de $L$-espaces vectoriels $\psi: \Lambda\to \Lambda$
tel que $\psi(\varphi(f))=f$ pour tout $f\in\Lambda$.

  Soit $t=\log(1+T)\in \mathcal{R}$. On a alors $\sigma_a(t)=at$ pour tout $a\in\zpet$ et $\varphi(t)=pt$. De plus, on
dispose pour tout $n\geq 1$ d'une injection $\Gamma$-\'{e}quivariante $\varphi^{-n}: \mathcal{E}^{]0,r_n]}\to L_n[[t]]$, qui
 envoie $f$ sur $f(\varepsilon^{(n)}e^{t/p^n}-1)$(noter que, comme $f$ converge en $\varepsilon^{(n)}-1$, $f(\varepsilon^{(n)}e^{t/p^n}-1)$
   est bien d\'{e}fini en tant qu'\'{e}l\'{e}ment de $L_n[[t]]$).

\subsection{Anneaux de Fontaine}\label{anneauxFontaine}

  Soit $\mathbf{C}_p$ le compl\'{e}t\'{e} de $\overline{\qp}$ et soit $O_{\mathbf{C}_p}$ l'anneau de ses entiers. Soit $v_p$ la valuation
 $p$-adique sur $\mathbf{C}_p$. On note $\mathbf{\tilde{E}}^+$ l'anneau des suites $x=(x^{(n)})_{n\geq 0}\in O_{\mathbf{C}_p}^{\mathbf{N}}$
telles que $(x^{(n+1)})^p=x^{(n)}$ pour tout
  $n$, l'addition \'{e}tant d\'{e}finie par $$(x+y)^{(n)}=\lim_{j\to\infty} (x^{(n+j)}+y^{(n+j)})^{p^j}$$ et la multiplication \'{e}tant
 d\'{e}finie composante par composante; c'est un anneau parfait de caract\'{e}ristique $p$, muni d'une action\footnote{Cette action est
induite par l'action tautologique de ${\rm Gal}(\overline{\qp}/\qp)$ sur $O_{\mathbf{C}_p}$.} de ${\rm Gal}(\overline{\qp}/\qp)$, commutant
 au Frobenius $x\to x^p$, que l'on note $\varphi$. En posant $v_E(x)=v_p(x^{(0)})$ pour $x\in \mathbf{\tilde{E}}^+$, on obtient une valuation
   sur l'anneau $\mathbf{\tilde{E}}^+$.
L'\'{e}l\'{e}ment $\overline{T}=(\varepsilon^{(n)})_n-1\in\mathbf{\tilde{E}}^+$ satisfait $v_E(\overline{T})=\frac{p}{p-1}$
et le corps des fractions $\mathbf{\tilde{E}}=\mathbf{\tilde{E}}^+[1/ \overline{T}]$ de $\mathbf{\tilde{E}}^+$ est alg\'{e}briquement
 clos et complet pour la valuation $v_E$.

   Soit $\mathbf{\tilde{A}}^+=W(\mathbf{\tilde{E}}^+)$ (resp. $\mathbf{\tilde{A}}=W(\mathbf{\tilde E})$) l'anneau des vecteurs de Witt \`{a}
coefficients dans $\mathbf{\tilde{E}}^+$ (resp. $\mathbf{\tilde{E}}$) et soit $\mathbf{\tilde{B}}^+=\mathbf{\tilde{A}}^+[1/p]$
 (resp. $\mathbf{\tilde{B}}=\mathbf{\tilde{A}}[1/p]$). Tout \'{e}l\'{e}ment de
  $\mathbf{\tilde{A}}^+$ (resp. $\mathbf{\tilde{A}}$) s'\'{e}crit de mani\`{e}re unique sous la forme
  $\sum_{k\geq 0} p^k[x_k]$, avec $x_k\in \mathbf{\tilde{E}}^+$ (resp. $\mathbf{\tilde{E}}$).
  Ici $[x]\in \mathbf{\tilde{A}}^+$ (resp. $\mathbf{\tilde{A}}$) est le repr\'{e}sentant de Teichmuller
de $x\in \mathbf{\tilde{E}}^+$ (resp. $\mathbf{\tilde{E}}$).
  Le Frobenius et l'action de ${\rm Gal}(\overline{\qp}/\qp)$ sur
  $\mathbf{\tilde{E}}^+$ et $\mathbf{\tilde{E}}$ induisent un Frobenius bijectif $\varphi$ et une action de
  ${\rm Gal}(\overline{\qp}/\qp)$ sur $\mathbf{\tilde{A}}^+$, $\mathbf{\tilde{A}}$, $\mathbf{\tilde{B}}$ et
   $\mathbf{\tilde{B}}^+$.
  Le r\'{e}sultat suivant, dans lequel $T=[1+\overline{T}]-1\in\mathbf{\tilde{A}}^+$, est standard~(\cite{FoAnnals}, propositions
  2.4, 2.12 et 2.17).

  \begin{proposition}\label{thetamap}
   L'application $\theta: \mathbf{\tilde{A}}^+\to O_{\mathbf{C}_p}$ d\'{e}finie par
     $\theta(x)=\sum_{n\geq 0} p^n x_n^{(0)}$ si $x=\sum_{n\geq 0} p^n[x_n]$ est un morphisme surjectif d'anneaux et ${\rm Ker}(\theta)=\omega\cdot \mathbf{ \tilde{A}}^+$,
     o\`{u} $\omega=\frac{T}{\varphi^{-1}(T)}$.
     Le s\'{e}par\'{e} compl\'{e}t\'{e} $\mathbf{B}_{\rm dR}^+$ de $\mathbf{\tilde{B}}^+$ pour la topologie $\omega$-adique est un
 anneau de valuation discr\`{e}te, d'uniformisante
 $\omega$ ou\footnote{Noter que l'on utilise la m\^{e}me lettre que
 pour l'\'{e}l\'{e}ment $t=\log(1+T)$ de $\mathcal{R}$.
 Cela s'explique par le fait que l'on a une injection naturelle de $\mathcal{R}^+=\mathcal{R}\cap L[[T]]$
 dans $\mathbf{B}_{\rm dR}^+$, qui envoie $f$ sur $f(e^t-1)$ et cette injection envoie l'\'{e}l\'{e}ment $t$ de
 $\mathcal{R}$ sur $t\in \mathbf{B}_{\rm dR}^+$.} $$t=\log(1+T)=\sum_{n\geq 1} (-1)^{n-1}\frac{T^n}{n}.$$

  \end{proposition}

 Le morphisme $\theta:\mathbf{\tilde{B}}^+\to\mathbf{C}_p$ s'\'{e}tend par
    continuit\'{e} en un morphisme surjectif $\mathbf{B}_{\rm{dR}}^+\to\mathbf{C}_p$, dont le
    noyau est engendr\'{e} par $\omega$ ou $t$. On note $\mathbf{B}_{\rm{dR}}=\mathbf{B}_{{\rm dR}}^+[1/t]$, qui est donc un
corps de valuation discr\`{e}te, d'uniformisante $t$ et de corps
    r\'{e}siduel $\mathbf{C}_p$.

    On dispose d'une myriade de sous-anneaux de $\mathbf{\tilde{B}}$, qui sont
   stables sous l'action de
    ${\rm Gal}(\overline{\qp}/\qp)$ et font le lien entre la th\'{e}orie de Hodge $p$-adique et la th\'{e}orie des $(\varphi,\Gamma)$-modules.

 a) Pour tout $r>0$, soit $\mathbf{\tilde{B}}^{(0,r]}$ l'ensemble des $x=\sum_{k>>-\infty} p^k[x_k]
   \in \mathbf{\tilde{B}}$ tels que $\lim_{k\to\infty} v_E(x_k)+\frac{k}{r}=\infty$. Soit $\mathbf{\tilde{B}}^{\dagger}$
   la r\'{e}union des $\mathbf{\tilde{B}}^{(0,r]}$ \`{a} l'int\'{e}rieur de $\mathbf{\tilde{B}}$. Comme $\varphi(\mathbf{\tilde{B}}^{(0,r]})=
   \mathbf{\tilde{B}}^{(0,r/p]}$, $\mathbf{\tilde{B}}^{\dagger}$ est stable
   sous l'action de $\varphi$, qui est bijectif.

 b) Si $x=\sum_{k>>-\infty} p^k[x_k]\in \mathbf{\tilde{B}}^{(0,r]}$, on note $$v^{(0,r]}(x)=\inf_{k\in\mathbf{Z}} \left(v_E(x_k)+\frac{k}{r}\right).$$
Soit $\mathbf{\tilde{B}}^{]0,r]}$
   la compl\'{e}tion de $\mathbf{\tilde{B}}^{(0,r]}$ pour la topologie de Fr\'{e}chet induite par les
   semi-valuations $(\min (v^{(0,s]}, v^{(0,r]}))_{0<s\leq r}$ et soit
   $\mathbf{\tilde{B}}_{\rm rig}$ la r\'{e}union des $\mathbf{\tilde{B}}^{]0,r]}$.
  Cet anneau est muni d'un Frobenius bijectif, obtenu par prolongement \`{a} partir
  du Frobenius sur $\mathbf{\tilde{B}}^{\dagger}$.

 c) Soit $\mathbf{A}_{\qp}$ l'anneau\footnote{Noter que $\mathbf{A}_{\qp}$ n'est rien d'autre que
l'anneau $O_{\mathcal{E}}$ pour $L=\qp$.} des s\'{e}ries de Laurent $\sum_{k\in \mathbf{Z}} a_k T^k$, o\`{u}
$a_k\in\zp$ et $\lim_{k\to -\infty} a_k=0$. On voit $\mathbf{A}_{\qp}$ \`{a} l'int\'{e}rieur de
    $\mathbf{\tilde{A}}$ en envoyant $T$ sur $[1+\overline{T}]-1$. Soit
    $\mathbf{B}$ l'adh\'{e}rence, pour la topologie $p$-adique,
    de l'extension
    maximale non-ramifi\'{e}e de $\mathbf{A}_{\qp}[1/p]$ dans $\mathbf{\tilde{B}}$. Ce sous-anneau de $\mathbf{\tilde{B}}$
est stable sous l'action de $\varphi$.
Finalement, on note
   $\mathbf{B}^{(0,r]}=\mathbf{B}\cap \mathbf{\tilde{B}}^{(0,r]}$
   et $\mathbf{B}_{\rm rig}=\mathcal{R}\otimes_{\mathcal{E}^{\dagger}}
   \mathbf{B}^{\dagger}$.

  On peut montrer\footnote{L'action
   de ${\rm Gal}(\overline{\qp}/\qp)$ sur tous les anneaux et objets qui suivent est $L$-lin\'{e}aire. Pour la preuve de ces isomorphismes,
 voir~\cite{Espvect}, propositions 7.1, 7.5 et 7.6, ainsi que~\cite{Ber}, proposition 3.15.}
que l'on a des identifications compatibles aux actions\footnote{Noter que les actions de $\varphi$ et $\Gamma$ sur $T=[1+\overline{T}]-1$
sont donn\'{e}es par $\varphi(T)=(1+T)^p-1$ et
$\sigma_a(T)=(1+T)^a-1$, i.e. ce sont les m\^{e}mes que celles d\'{e}finies dans la partie \ref{Fctan}.} de $\varphi$ et
 $\Gamma$ $$(L\otimes_{\qp} \mathbf{B})^H=\mathcal{E},\quad
   (L\otimes_{\qp} \mathbf{B}^{\dagger})^H=\mathcal{E}^{\dagger}, \quad
   (L\otimes_{\qp} \mathbf{B}_{\rm rig})^H=\mathcal{R}$$ et
   $(L\otimes_{\qp} \mathbf{B}^{(0,r]})^H=\mathcal{E}^{(0,r]}$ pour tout
   $r>0$ assez petit.

 \subsection{$(\varphi,\Gamma)$-modules}\label{phi}

 \begin{definition} \label{phimodules} Un $(\varphi,\Gamma)$-module $D$ sur $\Lambda\in \{\mathcal{E},\mathcal{E}^{\dagger},\mathcal{R}\}$
  est (dans cet article) un $\Lambda$-module libre de type fini, muni d'une action continue semi-lin\'{e}aire
  de $\Gamma$ et d'un op\'{e}rateur semi-lin\'{e}aire $\varphi$, commutant \`{a} $\Gamma$ et tel que $\Lambda\varphi(D)=D$.
 \end{definition}

  \textit{Dans la suite de ce chapitre, on fixe une $L$-repr\'{e}sentation $V$ de} ${\rm Gal}(\overline{\qp}/\qp)$. La
th\'{e}orie de Fontaine~\cite{FoGrot}, raffin\'{e}e par les travaux de
   Cherbonnier-Colmez~\cite{CCsurconv}, Kedlaya~\cite{Ked} et Berger~\cite{Ber}, associe \`{a} $V$ une famille
 de $\Gamma$ et $(\varphi,\Gamma)$-modules.

    \begin{definition} \label{phigamma}
    On d\'{e}finit:

    a) Les $(\varphi,\Gamma)$-modules $D=(\mathbf{B}\otimes_{\qp} V)^H$, $D^{\dagger}=(\mathbf{B}^{\dagger}\otimes_{\qp} V)^H$, $\tilde{D}_{\rm rig}=
    (\mathbf{\tilde{B}}_{\rm rig}\otimes_{\qp} V)^H,$
    $$D_{\rm rig}=(\mathbf{B}_{\rm rig}\otimes_{\qp} V)^H=\mathcal{R}\otimes_{\mathcal{E}^{\dagger}} D^{\dagger}$$ sur $\mathcal{E}$, $\mathcal{E}^{\dagger}$, $(L\otimes_{\qp} \mathbf{\tilde{B}}_{\rm rig})^H$ et $\mathcal{R}$, respectivement.

    b) Les $\Gamma$-modules $D^{(0,r_n]}=(\mathbf{B}^{(0,p^{-n}]}\otimes_{\qp} V)^H$, $\tilde{D}^{(0,r_n]}=(\mathbf{\tilde{B}}^{(0,p^{-n}]}\otimes_{\qp} V)^H$,
    $D^{]0,r_n]}=\mathcal{E}^{]0,r_n]}\otimes_{\mathcal{E}^{(0,r_n]}} D^{(0,r_n]}$
 et $\tilde{D}^{]0,r_n]}=(\mathbf{\tilde{B}}^{]0,p^{-n}]}\otimes_{\qp} V)^H$.

    c) Les $(\varphi,\Gamma)$-modules $\tilde{D}=(\mathbf{\tilde{B}}\otimes_{\qp} V)^H$, $\tilde{D}^+=(\mathbf{\tilde{B}}^+\otimes_{\qp} V)^H$.

    d) Les
    $\Gamma$-modules   $\tilde{D}_{\text{dif}}=(\mathbf{B}_{\text{dR}}\otimes_{\qp} V)^H$
    et $\tilde{D}_{\text{dif}}^+=(\mathbf{B}_{\text{dR}}^+\otimes_{\qp} V)^H$.

  \end{definition}

  On dispose alors du r\'{e}sultat fondamental suivant, d\^{u} \`{a} Cherbonnier et Colmez~\cite{CCsurconv}.

  \begin{theorem}\label{surconv}
   Il existe un entier $m(D)$ tel que
 $D^{(0,r_{m(D)}]}$ soit
   un $\mathcal{E}^{(0,r_{m(D)}]}$-module libre de rang $\dim_L V$, tel que\footnote{Tous les produits tensoriels sont pris au-dessus de $\mathcal{E}^{(0,r_{m(D)}]}$.}
   $D=\mathcal{E}\otimes D^{(0,r_{m(D)}]}$. De plus, pour tout $n\geq m(D)$ on a
   $D^{(0,r_n]}=
\mathcal{E}^{(0,r_n]}\otimes D^{(0,r_{m(D)}]}$.

  \end{theorem}

  \begin{remark}\label{deb} 1) On d\'{e}duit du th\'{e}or\`{e}me \ref{surconv} que $D^{]0,r_n]}$,
$D^{\dagger}$ et $D_{\rm rig}$ s'obtiennent \`{a} partir de $D^{(0,r_{m(D)}]}$ par extension des scalaires
\`{a} $\mathcal{E}^{]0,r_n]}$, $\mathcal{E}^{\dagger}$ et $\mathcal{R}$, respectivement. On obtient aussi
un isomorphisme canonique $D=\mathcal{E}\otimes_{\mathcal{E}^{\dagger}} D^{\dagger}$.

2) Dans la suite on se permettra d'augmenter un nombre fini de fois $m(D)$ si besoin est. Par exemple, dans le chapitre V de~\cite{Cbigone} on construit
   des entiers $m(D)$, $m_1(D)$,...,$m_5(D)$. On supposera que notre
   $m(D)$ est plus grand que toutes ces constantes.
  \end{remark}

\begin{remark}\label{pleinefidelite}
 On peut r\'{e}cup\'{e}rer $V$ de mani\`{e}re fonctorielle \`{a} partir de $D$, $D^{\dagger}$ et $D_{\rm rig}$
 par $$V=((L\otimes_{\qp} \mathbf{B})\otimes_{\mathcal{E}} D)^{\varphi=1}= ((L\otimes_{\qp} \mathbf{B}^{\dagger})\otimes_{\mathcal{E}^{\dagger}} D^{\dagger})^{\varphi=1}=((L\otimes_{\qp}\mathbf{\tilde{B}}_{\rm rig})\otimes_{\mathcal{R}} D_{\rm rig})^{\varphi=1}.$$
Comme on a aussi un isomorphisme fonctoriel $D_{\rm rig}=(\mathbf{B}_{\rm rig}\otimes_{\qp} V)^H$, le foncteur $V\to D_{\rm rig}$ est
pleinement fid\`{e}le; en particulier ${\rm End}_{L[{\rm Gal}(\overline{\qp}/\qp)]}(V)={\rm End}_{\varphi,\Gamma, \mathcal{R}}(D_{\rm rig})$
pour toute $L$-repr\'{e}sentation $V$.

  \end{remark}

  On finit cette partie en rappelant la construction des applications
  de localisation pour un $(\varphi,\Gamma)$-module. Elles vont jouer un
  r\^{o}le crucial dans la suite. Si $x=\sum_{k>>-\infty}p^k[x_k]\in \mathbf{\tilde{B}}^{(0,1]}$, la s\'{e}rie
   $\sum_{k>>-\infty} p^k[x_k]$
     converge dans $\mathbf{B}_{\text{dR}}^+$, ce qui fournit un morphisme naturel $\mathbf{\tilde{B}}^{(0,1]}\to
     \mathbf{B}_{\text{dR}}^+$, qui se trouve \^{e}tre injectif. Ce morphisme s'\'{e}tend en un morphisme, toujours injectif, de $\mathbf{\tilde{B}}^{]0,1]}$ dans $\mathbf{B}_{\text{dR}}^+$.
     Compos\'{e} avec l'isomorphisme $\mathbf{\tilde{B}}^{]0,1]}\simeq \mathbf{\tilde{B}}^{]0,p^{-n}]}$
     induit par $\varphi^n$, ce morphisme induit une application de localisation
     $\varphi^{-n}: \mathbf{\tilde{B}}^{]0,p^{-n}]}\to \mathbf{B}_{\text{dR}}^+$, compatible avec l'action de ${\rm Gal}(\overline{\qp}/\qp)$.

  \begin{definition}
  \label{ddif}
On note $\varphi^{-n}:
D^{]0,r_n]}\to \tilde{D}_{\rm dif}^+$ la compos\'{e}e de l'injection
$D^{]0,r_n]}\subset \tilde{D}^{]0, r_n]}=(\mathbf{\tilde{B}}^{]0,p^{-n}]}\otimes_{\qp}V)^H$ avec l'application $\varphi^{-n}\otimes 1: \mathbf{\tilde{B}}^{]0,p^{-n}]}\otimes V\to \mathbf{B}_{\text{dR}}^+\otimes V$. Cette injection $L$-lin\'{e}aire, $\Gamma$-\'{e}quivariante est appel\'{e}e \textit{morphisme de localisation} en $\varepsilon^{(n)}-1$.
\end{definition}

\begin{example}\label{trivialite} Si $D$ est le $(\varphi,\Gamma)$-module
   attach\'{e} \`{a} la repr\'{e}sentation triviale, le morphisme de localisation est l'injection $\varphi^{-n}$ de $\mathcal{E}^{]0,r_n]}\to L_n[[t]]$
   de la partie \ref{Fctan}, compos\'{e}e avec l'inclusion $L_n[[t]]\subset
    L\otimes_{\qp} (\mathbf{B}_{\rm dR}^+)^H$.
 \end{example}

\subsection{Le module de Fontaine $D_{{\rm dif},n}^+$}

  \begin{definition} Soit $D_{{\rm dif},n}^+$ le sous
   $L_n[[t]]$-module de $\tilde{D}_{\rm dif}^+$ engendr\'{e} par l'image du morphisme de localisation $\varphi^{-n}: D^{]0,r_n]}\to\tilde{D}_{\rm dif}^+=(\mathbf{B}_{\rm dR}^+\otimes_{\qp} V)^H$.
  \end{definition}

 \begin{remark}\label{moduleFontaine}

  a) Pour $n\geq m(D)$ on a
   $$D_{{\rm dif},n}^+=L_n[[t]]\otimes_{\mathcal{E}^{]0,r_n]}} D^{]0,r_n]},$$
   o\`{u} $L_n[[t]]$ est vu comme $\mathcal{E}^{]0,r_n]}$-module via $\varphi^{-n}$ (voir l'exemple
   \ref{trivialite}).
  Le th\'{e}or\`{e}me \ref{surconv} et la remarque \ref{deb} entra\^{i}nent que
 $D_{{\rm dif},n}^+$ est un $L_n[[t]]$-module libre de rang $\dim_L V$,
 stable sous l'action de $\Gamma$. De plus, $\Gamma_n$ agit de mani\`{e}re
 $L_n$-lin\'{e}aire sur $D_{{\rm dif},n}^+$.

b) La repr\'{e}sentation $V$ est dite de de Rham si le
   $L$-espace vectoriel $$D_{\rm dR}(V)=(\mathbf{B}_{\rm dR}\otimes_{\qp} V)^{{\rm Gal} (\overline{\qp}/\qp)}$$ est de dimension $\dim_L V$. La filtration
   sur $D_{\rm dR}(V)$ est d\'{e}finie par ${\rm Fil}^{i}(D_{\rm dR}(V))=(t^i\mathbf{B}_{\rm dR}^+\otimes_{\qp} V)^{{\rm Gal} (\overline{\qp}/\qp)}$ pour
   tout $i\in \Z$. On note aussi $D^+_{\rm dR}(V)={\rm Fil}^0(D_{\rm dR}(V))$.
 Si $V$ est de de Rham, alors on peut retrouver
$D_{\rm dR}(V)$ et sa filtration \`{a} partir de $D_{{\rm dif},n}^+$ par la recette
${\rm Fil}^{i}(D_{\rm dR}(V))=(t^iD_{{\rm dif},n}^+)^{\Gamma}$, pour $n\geq m(D)$ et $i\in \Z$. De plus, $D_{{\rm dif},n}^+[1/t]=L_n((t))\otimes_{L} D_{\rm dR}(V)$
pour tout $n\geq m(D)$. Notons qu'avec nos conventions, si $V$ est de de Rham, alors les poids de Hodge-Tate de $V$ sont les oppos\'{e}s des sauts de la filtration sur $D_{\rm dR}(V)$.

c) Dans tous les cas, $D_{{\rm dif},n}^+$ engendre $\mathbf{B}_{\rm dR}^+\otimes_{\qp} V$ comme $\mathbf{B}_{\rm dR}^+$-module (voir \cite{FoAst}).

\end{remark}

 \subsection{L'action infinit\'{e}simale de $\Gamma$}\label{FontaineSen}

  D'apr\`{e}s Berger~\cite{Ber}, lemme 4.1 et ce qui suit, l'action infinit\'{e}simale de
  $\Gamma$ sur $D_{\rm rig}$ d\'{e}finit une connexion
     $$\nabla: D_{{\rm rig}}\to D_{{\rm rig}}, \quad \nabla(z)=\lim_{a\to 1}\frac{\sigma_a(z)-z}{a-1},$$ qui commute \`{a} $\varphi$ et $\Gamma$
      et laisse stable $D^{]0,r_n]}$ pour $n\geq m(D)$.
    Par exemple, si $V$ est triviale, alors $D_{\rm rig}=\mathcal{R}$
    et $$\nabla(f)=(1+T)\log(1+T)\cdot \frac{d}{dT}f(T)=t\cdot\frac{df}{dt}$$ pour $f\in\mathcal{R}$.
     En g\'{e}n\'{e}ral, la connexion
    $\nabla$ satisfait
     $\nabla(fz)=\nabla(f)z+f\nabla(z)$ pour tout $z\in D_{\rm rig}$ et tout
     $f\in\mathcal{R}$.

   L'action de $\Gamma$ sur $D_{{\rm dif},n}^+$ (avec $n\geq m(D)$) peut aussi se d\'{e}river, ce qui induit une connexion $L_n$-lin\'{e}aire $\nabla=\lim_{a\to 1}\frac{\sigma_a-1}{a-1}$ sur $D_{{\rm dif},n}^+$, au-dessus de la connexion $t\frac{d}{dt}$ sur
   $L_n[[t]]$. Cette connexion pr\'{e}serve donc $t^i D_{{\rm dif},n}^+$ pour tout
   $i\geq 0$.

  \begin{definition}\label{HTdef} Les poids de Hodge-Tate
  g\'{e}n\'{e}ralis\'{e}s (resp. le polyn\^{o}me $P_{\rm Sen,V}$ de Sen) de $V$ sont les valeurs propres (resp. le polyn\^{o}me caract\'{e}ristique) de $\nabla$ agissant sur le $L_n$-module libre de type fini $D_{{\rm Sen},n}:=D_{{\rm dif},n}^+/tD_{{\rm dif},n}^+$.
  \end{definition}

\begin{remark}
 La d\'{e}finition \ref{HTdef} est compatible avec la d\'{e}finition classique\footnote{D'apr\`{e}s Sen~\cite{Sen}, il existe un plus grand
  sous-$L_{\infty}$-module libre de type
  fini $D_{\rm Sen}(V)$ de $(\mathbf{C}_p\otimes_{\qp} V)^H$,
  qui est stable par $\Gamma$. Le module $D_{\rm Sen}(V)$ engendre $\mathbf{C}_p\otimes_{\qp} V$ sur $\mathbf{C}_p$ et est de rang $\dim_L V$. L'action infinit\'{e}simale de $\Gamma$ sur
  $D_{\rm Sen}(V)$ d\'{e}finit un op\'{e}rateur $L_{\infty}$-lin\'{e}aire $\Theta_{\rm Sen}$, dont les valeurs propres s'appelent les poids de Hodge-Tate g\'{e}n\'{e}ralis\'{e}s de $V$.} des poids de Hodge-Tate g\'{e}n\'{e}ralis\'{e}s. En effet, $\theta: (\mathbf{B}_{\rm dR}^+\otimes_{\qp} V)^H
 \to (\mathbf{C}_p\otimes_{\qp} V)^H$ identifie $D_{{\rm Sen},n}$ \`{a} un sous-$L_n$-module de $D_{\rm Sen}(V)$, tel que
 $D_{{\rm Sen},n}\otimes_{L_n} L_{\infty}=D_{\rm Sen}(V)$.

\end{remark}

 Le r\'{e}sultat suivant joue un r\^{o}le crucial dans la preuve du th\'{e}or\`{e}me
 \ref{main1}.

\begin{proposition}\label{HC}
 Soit $V$ une $L$-repr\'{e}sentation de ${\rm Gal}(\overline{\qp}/\qp)$ et
 soit $D_{\rm rig}$ le $(\varphi,\Gamma)$-module sur $\mathcal{R}$ attach\'{e} \`{a}
 $V$. Alors $P_{\rm{Sen,V}}(\nabla)$ envoie $D_{\rm rig}$ dans $t\cdot D_{\rm rig}$.
\end{proposition}

\begin{proof}
 Si $z\in D_{\rm rig}$ et $\varphi^{-n}(z)\in tD_{\rm dif, n}^+$ pour
    tout $n$ assez grand, alors $z\in tD_{\rm rig}$ (voir~\cite{Ber}, lemmes 5.1
    et 5.4). On se ram\`{e}ne donc \`{a} montrer que pour tout $z\in D_{\rm rig}$ et pour tout $n$ assez grand on a
    $\varphi^{-n}(P_{\rm Sen, V}(\nabla)(z))\in t\cdot D_{{\rm dif},n}^+$.
    Puisque $\varphi^{-n}$ commute \`{a} $\nabla$, il suffit de v\'{e}rifier que
    $P_{\rm Sen, V}(\nabla)$ envoie $D_{\rm dif, n}^+$ dans $t\cdot D_{{\rm dif},n}^+$.
    Cela d\'{e}coule du th\'{e}or\`{e}me de Cayley-Hamilton dans le $L_n$-module libre de type fini $D_{{\rm Sen},n}$.

\end{proof}

\subsection{Traces de Tate normalis\'{e}es}\label{Tate}

   Soit $L_{\infty}=\cup_{n} L_n$. Pour tout
   $n\geq 0$ on dispose d'un projecteur $\Gamma$-\'{e}quivariant $T_n: L_{\infty}\to L_n$,
   d\'{e}fini par $$T_n(x)=\frac{1}{p^j}{\rm Tr}_{L_{n+j}/L_n}(x)$$ pour tout $j$ tel que $x\in L_{n+j}$.
   On en d\'{e}duit des projecteurs $\Gamma$-\'{e}quivariants
   $${\rm Res}_{p^{-n}\zp}: L_{\infty}((t))\to L((t)), \quad {\rm Res}_{p^{-n}\zp}(\sum_{l>>-\infty} a_l t^l)=
   \sum_{l>>-\infty} T_n(a_l)t^l.$$
     Dans~\cite{Cannals}, proposition V.4.5,
    Colmez \'{e}tend ces projecteurs en des
   projecteurs $L((t))$-lin\'{e}aires continus (appel\'{e}s
   \textit{ traces de Tate normalis\'{e}es})
   ${\rm Res}_{p^{-n}\zp}: L\otimes \mathbf{B}^H_{\rm dR}\to L_n((t))$, envoyant $L\otimes (\mathbf{B}_{\rm dR}^+)^H$ dans
   $L_n[[t]]$.

\subsection{L'op\'{e}rateur $\psi$}

    Comme
    $\varphi(D)$ engendre $D$,
    tout \'{e}l\'{e}ment $z$ de $D$ s'\'{e}crit de mani\`{e}re unique
    sous la forme $$z=\sum_{i=0}^{p-1}(1+T)^i\varphi(z_i),$$ avec
    $z_i\in D$, ce qui permet de d\'{e}finir un op\'{e}rateur
    continu $\psi$ sur $D$ par $\psi(z)=z_0$. On v\'{e}rifie facilement que
    $\psi(f\varphi(z))=\psi(f)z$ et $\psi(\varphi(f)z)=f\psi(z)$
    pour $f\in \mathcal{E}$ et $z\in D$. De plus, $\psi$ commute \`{a} $\Gamma$.
    Tous ces r\'{e}sultats restent valables si on remplace $D$ par $D_{\rm rig}$ et $\mathcal{E}$ par $\mathcal{R}$.

\section{Actions infinit\'{e}simales sur $D_{\rm rig}$}

  Dans ce chapitre on d\'{e}montre le th\'{e}or\`{e}me \ref{main1}. L'id\'{e}e de la preuve est tr\`{e}s simple: on montre que l'op\'{e}rateur de Casimir (qui engendre le centre de
  l'alg\`{e}bre $U(\mathfrak{sl_2})$) induit un endomorphisme de
  $D_{\rm rig}$, qui est forc\'{e}ment scalaire par des r\'{e}sultats
  de Berger et Kedlaya. Pour identifier ce scalaire, on utilise le lien
  entre la th\'{e}orie de Hodge $p$-adique et les $(\varphi,\Gamma)$-modules,
  plus pr\'{e}cis\'{e}ment la proposition \ref{HC}. Pour en d\'{e}duire
  l'action de $u^-$ sur $D_{\rm rig}$ il ne reste plus qu'\`{a} revenir
  aux formules explicites donnant l'action de ${\rm GL}_2(\qp)$ et
  \`{a} exprimer l'action du Casimir seulement en fonction de $u^-$.
  Les premi\`{e}res parties de ce chapitre sont pr\'{e}liminaires
  et rappellent des constructions diverses et vari\'{e}es de Colmez.
  \textit{Dans la suite $V$ est une $L$-repr\'{e}sentation
  absolument irr\'{e}ductible, de dimension $2$ de ${\rm Gal}(\overline{\qp}/\qp)$.}
  On note $D$ et $D_{\rm rig}$ les $(\varphi,\Gamma)$-modules sur
  $\mathcal{E}$ et $\mathcal{R}$ attach\'{e}s \`{a} $V$ (d\'{e}finition
  \ref{phigamma}).

\subsection{Le faisceau $U\to \Delta\boxtimes U$}\label{involution}

   Soit $\Delta\in\{D,D_{\rm rig}\}$. Dans la partie \ref{faisceau}, on a expliqu\'{e} la construction d'un faisceau $P^+$-\'{e}quivariant
   sur $\zp$, attach\'{e} \`{a} $\Delta$, qui se prolonge (d'apr\`{e}s les chapitres II et V de~\cite{Cbigone})
    en un faisceau sur ${\p1}(\qp)$,
 \'{e}quivariant sous l'action de ${\rm GL}_2(\qp)$. Soit $w_D$ la restriction \`{a} $\Delta\boxtimes\zpet$ de l'action de l'involution $w=\left(\begin{smallmatrix} 0 & 1 \\1 & 0 \end{smallmatrix}\right)$, de telle sorte que
 l'espace des sections globales du faisceau
  attach\'{e} \`{a} $\Delta$ est
  $$\Delta\boxtimes\p1=\{(z_1,z_2)\in \Delta\times \Delta| \quad {\rm Res}_{\zpet}(z_2)=
w_D({\rm Res}_{\zpet}(z_1))\}.$$

L'application de prolongement par z\'{e}ro $i_{\zp}: \Delta\to \Delta\boxtimes
\p1$ d\'{e}finie par $i_{\zp}(z)=(z, w_D({\rm Res}_{\zpet}(z)))$ permet d'identifier
$\Delta$ \`{a} un sous-module de $\Delta\boxtimes \p1$, ce que l'on fera sans plus de commentaire. Tout \'{e}l\'{e}ment $z$ de $\Delta\boxtimes\p1$ peut alors s'\'{e}crire $z=z_1+w\cdot z_2$ avec $z_1,z_2\in \Delta$. Rappelons que
$\delta_D=\chi^{-1}\cdot \det V$ est vu comme caract\`{e}re de $\qpet$.
 L'action de ${\rm GL}_2(\qp)$ sur $\Delta=\Delta\boxtimes \zp$ est donn\'{e}e par les formules
suivantes:

a) $\left(\begin{smallmatrix} a & 0 \\0 & a\end{smallmatrix}\right)z=\delta_D(a)z$ et $\left(\begin{smallmatrix} p & 0 \\0 & 1\end{smallmatrix}\right)z=\varphi(z)$.

b) Si $a\in\zpet$, alors $\left(\begin{smallmatrix} a & 0 \\0 & 1\end{smallmatrix} \right)z=\sigma_a(z)$.

c) Pour tout $b\in p\zp$ on a $\left(\begin{smallmatrix} 1 & b \\0 & 1\end{smallmatrix}\right)z=(1+T)^b\cdot z$ et
$$\left(\begin{smallmatrix} 1 & 0 \\b & 1\end{smallmatrix}\right)z=w_D\left((1+T)^b\cdot w_D\left({\rm Res}_{\zpet}(z)\right)\right)+u_b\left({\rm Res}_{p\zp}(z)\right),$$ avec\footnote{ La formule de~\cite{Cbigone} p. 325 comporte quelques erreurs de frappe.}, pour $x\in\varphi(D)$,
  $$u_b(x)=\delta_D^{-1}(1-b)(1+T)^{\frac{1}{b-1}} w_D[(1+T)^{b(1-b)}\sigma_{(1-b)^2}(w_D((1+T)x))].$$

\begin{remark}
 Les formules donnant l'action du Borel sup\'{e}rieur sur $\Delta=\Delta\boxtimes \zp$ sont facile \`{a} manipuler et expliquent les formules simples pour l'action
 de ${\rm I}_2$, $u^+$ et $h$ dans le th\'{e}or\`{e}me \ref{main1}. Par contre, la formule donnant l'action de l'unipotent
 inf\'{e}rieur ne permet pas une d\'{e}duction facile de la formule pour $u^-$.

\end{remark}

\begin{remark}\label{explic}
 Les formules donnant l'action de ${\rm GL}_2(\qp)$ sont inspir\'{e}es
de celles obtenues en faisant agir ${\rm GL}_2(\qp)$ sur les mesures sur $\p1(\qp)$.
Rappelons que $L\otimes_{O_L} O_L[[T]]$ (resp. $\mathcal{R}^+=\mathcal{R}\cap L[[T]]$) s'identifie, gr\^{a}ce au th\'{e}or\`{e}me de Mahler (resp. Amice), \`{a} l'espace des mesures (resp. distributions) $\mu$ sur $\zp$, \`{a} valeurs dans $L$. L'identification se fait par la transform\'{e}e d'Amice $$\mu\to \mathcal{A}_{\mu}=\sum_{n\geq 0} \int_{\zp} \binom{x}{n}\mu\cdot T^n=\int_{\zp} (1+T)^x\mu.$$ Le mono\"{i}de $P^+$ agit sur l'espace de ces mesures (resp. distributions) par
$$\int_{\zp} \phi(x) g\cdot \mu=\int_{\zp} \phi(ax+b)\mu$$ si $\phi:\zp\to L$ est continue (resp. localement analytique) et $g=\left(\begin{smallmatrix} a & b \\0 & 1\end{smallmatrix}\right)$. Au niveau des transform\'{e}es d'Amice, l'action de $P^+$ sur l'espace des mesures (resp. distributions)
peut aussi s'\'{e}crire $$\mathcal{A}_{\left(\begin{smallmatrix} p^ka & b \\0 & 1\end{smallmatrix}\right)\mu}=(1+T)^b\cdot \varphi^k(\sigma_a(\mathcal{A}_{\mu})),$$ formule qui a un sens pour n'importe quel $(\varphi,\Gamma)$-module et qui a \'{e}t\'{e} utilis\'{e}e dans l'introduction pour d\'{e}finir une action de $P^+$ sur $\Delta$. Si $\mu$ est une mesure sur $\zp$ et $U$ est un ouvert compact de $\zp$, soit
${\rm Res}_U(\mu)$ la mesure telle que $\int_{\zp}\phi {\rm Res}_U(\mu)=\int_{\zp} 1_{U}\cdot \phi \mu$, o\`{u} $1_U$ est la fonction caract\`{e}ristique de
$U$. Si $a\in\zp$, $n\geq 0$ et $\Lambda\in \{L\otimes_{O_L} O_L[[T]],\mathcal{R}\}$, on d\'{e}finit ${\rm Res}_{a+p^n\zp}:\Lambda\to \Lambda$ par
${\rm Res}_{a+p^n\zp}=\left(\begin{smallmatrix} 1 & a \\0 & 1\end{smallmatrix}\right)\circ \varphi^n\circ \psi^n\circ
\left(\begin{smallmatrix} 1 & -a \\0 & 1\end{smallmatrix}\right)$. Tout ouvert compact $U$ de $\zp$ peut s'\'{e}crire comme une r\'{e}union disjointe finie
$U=\cup_{a\in I}(a+p^n\zp)$ et on pose ${\rm Res}_{U}=\sum_{a\in I} {\rm Res}_{a+p^n\zp}$, qui ne d\'{e}pend pas du choix de la d\'{e}composition.
On v\'{e}rifie alors que ${\rm Res}_{U}(\mathcal{A}_{\mu})=\mathcal{A}_{{\rm Res}_U(\mu)}$. Le faisceau \'{e}quivariant attach\'{e} \`{a} $D$ est l'analogue du faisceau des mesures sur $\zp$ et de l'action de $P^+$ que l'on vient de d\'{e}finir. L'espace $D\boxtimes\p1$
est alors l'analogue de l'espace des mesures sur
$\p1(\qp)$, muni de l'action de
${\rm GL}_2(\qp)$ d\'{e}finie par
 $$\int_{\p1(\qp)} \phi(x) g\cdot \mu =\int_{\p1(\qp)} \delta_D(cx+d)\phi\left(\frac{ax+b}{cx+d}\right)\mu$$
pour $g=\left(\begin{smallmatrix} a & b \\c & d\end{smallmatrix}\right)$ et
 $\phi:\zp\to L$ continue. L'espace $\Delta\boxtimes \zpet=\Delta^{\psi=0}$
correspond \`{a} l'espace des mesures \`{a} support dans $\zpet$ et l'involution $w_D$ de
$D\boxtimes\zpet$ correspond au niveau des mesures \`{a} la transform\'{e}e
$\mu\to w_D\cdot \mu$ d\'{e}finie par
  $$\mathcal{A}_{w_D\cdot \mu}=\int_{\zpet} \delta_D(x)(1+T)^{\frac{1}{x}}\mu.$$
 En approximant l'int\'{e}grale par des sommes de Riemann et developpement de Taylor
 \`{a} l'ordre $1$, on obtient une formule qui ne fait intervenir que l'action de
 $\varphi$, $\psi$ et $\Gamma$ sur $\mathcal{A}_{\mu}$. Cette formule a un sens pour tout
 $(\varphi,\Gamma)$-module $D$ et est donn\'{e}e par
    $$w_D(z)=\lim_{n\to\infty} \sum_{i\in \left(\mathbf{Z}/p^n\mathbf{Z}\right)^*} \delta_D(i)(1+T)^{\frac{1}{i}}\sigma_{-\frac{1}{i^2}}
    \varphi^n \psi^n ((1+T)^{-i}z)$$
pour $z\in D^{\psi=0}$. Les r\'{e}sultats g\'{e}n\'{e}raux de~\cite{Cmirab}, chap. V
montrent que la limite existe et que $w_D$ est une involution continue de
  $D^{\psi=0}$. La s\'{e}rie d\'{e}finissant $w_D(z)$ ne converge pas dans
  $D^{\dagger}$ ou $D_{\rm rig}$. Cependant, si $D$ est attach\'{e} \`{a} une repr\'{e}sentation
  de dimension $2$, Colmez d\'{e}montre\footnote{Ce r\'{e}sultat est hautement nontrivial et d\'{e}coule essentiellement
   de la construction de la correspondance de Langlands locale $p$-adique, par prolong\'{e}ment analytique \`{a} partir du cas cristallin.} (\cite{Cbigone}, lemme V.2.4 et proposition V.2.9) que $w_D$ pr\'{e}serve $D^{\dagger,\psi=0}$
   et se prolonge
  en une involution continue de $D_{\rm rig}^{\psi=0}$. Cela permet de faire les m\^{e}mes constructions pour $D_{\rm rig}$.

\end{remark}

\subsection{Actions infinit\'{e}simales sur le faisceau $U\to D_{\rm rig}\boxtimes U$}

 Si $m\geq 2$, soit $K_m=1+p^m M_2(\zp)$; c'est un groupe de Lie $p$-adique compact. Soit $\mathcal{D}(K_m)$
 l'alg\`{e}bre des distributions
  sur $K_m$ (voir~\cite{STJAMS}, ~\cite{STInv}), dual fort
   de l'alg\`{e}bre $C^{\rm an}(K_m,L)$ des fonctions localement analytiques sur $K_m$, \`{a} valeurs
   dans $L$. Soit $U(\mathfrak{gl_2})$ l'alg\`{e}bre enveloppante de l'alg\`{e}bre de Lie de ${\rm GL}_2(\qp)$. On dispose d'une inclusion $U(\mathfrak{gl_2})\subset \mathcal{D}(K_m,L)$, en voyant
les \'{e}l\'{e}ments
     de $U(\mathfrak{gl_2})$ comme des op\'{e}rateurs
     diff\'{e}rentiels sur $C^{\rm an}(K_m,L)$ et en \'{e}valuant en $1\in K_m$.

    \begin{proposition}
    \label{distrib}
    L'action de
    ${\rm GL}_2(\qp)$ sur $D_{{\rm rig}}\boxtimes \p1$ en fait un module sur
    $\mathcal{D}(K_m)$.
   \end{proposition}

\begin{proof}
  Passer \`{a} la limite inductive dans le lemme V.2.15 de~\cite{Cbigone} ou contempler
  le dernier isomorphisme de la remarque V.2.19 de~\cite{Cbigone}. Dans loc.cit., $m$
  est suppos\'{e} assez grand, mais comme les $K_m$ sont commensurables, cela ne change rien.
\end{proof}

\begin{lemma}
\label{distrlocal}
 Soit $H\subset {\rm GL}_2(\qp)$ un sous-groupe ouvert compact
qui stabilise l'ouvert compact $U\subset \p1(\qp)$. Alors
$D_{\rm rig}\boxtimes U\subset D_{\rm rig}\boxtimes \p1$
 est stable par $\mathcal{D}(H)$ et on a ${\rm Res}_U(\lambda\cdot z)=\lambda\cdot {\rm Res}_U(z)$
 pour tout $z\in D_{\rm rig}\boxtimes \p1$ et tout $\lambda\in \mathcal{D}(H)$.

\end{lemma}

\begin{proof}
  Si $\lambda\in L[H]$ est une combinaison lin\'{e}aire de masses de Dirac, cela
 est \'{e}vident par ${\rm GL}_2(\qp)$-\'{e}quivariance de ${\rm Res}_U$ et le fait que
 $h(U)=U$ pour $h\in H$. Le r\'{e}sultat suit alors par continuit\'{e} et densit\'{e}
 de $L[H]$ dans $\mathcal{D}(H)$~(\cite{STJAMS}, lemme 3.1).
\end{proof}

 Le lemme pr\'{e}c\'{e}dent induit une action de $\mathfrak{gl_2}$ sur $D_{\rm rig}\boxtimes \zp=D_{\rm rig}$, qui satisfait ${\rm Res}_{U}(X\cdot z)=X\cdot {\rm Res}_{U}(z)$ pour $z\in D_{\rm rig}\boxtimes \p1$, $X\in \mathfrak{gl_2}$ et $U\subset \zp$ ouvert compact.

\subsection{L'action de l'\'{e}l\'{e}ment de Casimir}

   Rappelons que l'on utilise la base $$I_2=\left(\begin{smallmatrix} 1 & 0 \\0 & 1\end{smallmatrix}\right),\quad h=\left(\begin{smallmatrix} 1 & 0 \\0 & -1\end{smallmatrix}\right),\quad
   u^+=\left(\begin{smallmatrix} 0 & 1 \\0 & 0\end{smallmatrix}\right), \quad u^-=\left(\begin{smallmatrix} 0 & 0 \\1 & 0\end{smallmatrix}\right)$$
   de $\mathfrak{gl_2}$.
 Pour la d\'{e}finition des poids de Hodge-Tate g\'{e}n\'{e}ralis\'{e}s d'une $L$-repr\'{e}sentation
 $V$ et du polyn\^{o}me $P_{\rm Sen,V}$, voir la partie \ref{FontaineSen}.

  \begin{proposition}
\label{trivial}
   Si $k$ est la somme des poids de Hodge-Tate g\'{e}n\'{e}ralis\'{e}s de $V$, alors pour tout $z\in D_{\rm rig}$ on a $I_2(z)=(k-1)z$, $u^+(z)=tz$ et $h(z)=(1-k)z+2\nabla z$.

  \end{proposition}

\begin{proof} La ${\rm GL}_2(\qp)$-repr\'{e}sentation $D_{\rm rig}\boxtimes \p1$ a pour caract\`{e}re central
$\delta_D=\chi^{-1}\cdot \det V$, qui, en tant que caract\`{e}re galoisien, a poids de Hodge-Tate g\'{e}n\'{e}ralis\'{e} $k-1$ (car $\det V$ a poids de Hodge-Tate $k$). Donc, en tant que caract\`{e}re
de $\qpet$, la d\'{e}riv\'{e}e en $1$ de $\delta_D$ est $k-1$. La formule $I_2(z)=(k-1)z$
s'en d\'{e}duit.
Comme $\left(\begin{smallmatrix} 1 & b \\0 & 1\end{smallmatrix}\right)z=(1+T)^b z$ pour $z\in D_{{\rm rig}}$ et $b\in p\zp$, on a
$$u^+(z)=\lim_{b\to 0}\frac{(1+T)^b-1}{b}z=tz$$
pour tout $z\in D_{\rm rig}$.
 Le caract\`{e}re central de $D_{\rm rig}\boxtimes\p1$ \'{e}tant $\delta_D$, un calcul imm\'{e}diat
donne
  $$\left(\begin{smallmatrix} a & 0 \\0 & a^{-1}\end{smallmatrix}\right)z=
  \delta_D^{-1}(a)\sigma_{a^2}(z).$$ Il s'ensuit que
  pour $z\in D_{{\rm rig}}$ on a
 $$h(z)=
 \lim_{a\to 1} \frac{\delta_{D}^{-1}(a)\sigma_{a^2}z-z}{a-1}=(1-k)z+2\nabla z,$$
 la derni\`{e}re \'{e}galit\'{e} \'{e}tant une cons\'{e}quence du fait que la d\'{e}riv\'{e}e de $\delta_D$ en $1$
 est $k-1$, comme on l'a expliqu\'{e} ci-dessus.

\end{proof}

 L'\'{e}l\'{e}ment de Casimir
    $$C=u^+u^-+u^-u^++\frac{1}{2}h^2$$ engendre le centre
    de $U(\mathfrak{sl_2})$. D'apr\`{e}s le lemme \ref{distrlocal} et la remarque qui le suit, on peut
    voir $C$ comme endomorphisme $L$-lin\'{e}aire continu de $D_{\rm rig}=D_{\rm rig}\boxtimes \zp$.

\begin{lemma}\label{Casimir}
 Soient $a$ et $b$ les poids de Hodge-Tate g\'{e}n\'{e}ralis\'{e}s de $V$, de telle sorte que
 $P_{\rm Sen, V}(X)=(X-a)(X-b)$. On a une \'{e}galit\'{e} d'op\'{e}rateurs sur $D_{\rm rig}$
  $$C=2tu^-+2P_{\rm Sen,V}(\nabla)+\frac{(a-b)^2-1}{2}.$$

\end{lemma}

\begin{proof}
 La proposition \ref{trivial} montre qu'en tant qu'op\'{e}rateurs sur $D_{\rm rig}$ on a
 $h=2\nabla+1-k$ et $u^+(z)=tz$, o\`{u} $k=a+b$. Un calcul
 imm\'{e}diat montre que
  $$\frac{1}{2}h^2-h=\frac{(2\nabla-k)^2-1}{2}=2P_{\rm Sen,V}(\nabla)+\frac{(a-b)^2-1}{2}.$$
 Le r\'{e}sultat suit alors de l'\'{e}galit\'{e} $C=2u^+u^-+\frac{1}{2}h^2-h$, d\'{e}duite de $u^+u^--u^-u^+=h$.

\end{proof}

  On est maintenant en mesure de d\'{e}montrer le
  th\'{e}or\`{e}me principal de l'article. On garde les notations du lemme
  \ref{Casimir}.

  \begin{theorem}\label{Casimirinf}
   Pour tout $z\in D_{\rm rig}$ on a $u^-(z)=-\frac{P_{\rm Sen, V}(\nabla)(z)}{t}$
   et $C(z)=\frac{(a-b)^2-1}{2}\cdot z$.
  \end{theorem}

\begin{proof}

  Comme $C$ commute \`{a} l'action adjointe de ${\rm GL}_2(\qp)$, les formules
  donnant l'action de ${\rm GL}_2(\qp)$ sur $D_{\rm rig}\boxtimes \p1$ montrent
  que l'op\'{e}rateur $C:D_{\rm rig}\to D_{\rm rig}$ commute \`{a}
  $\varphi$ et $\Gamma$ et satisfait $C((1+T)^m\cdot z)=(1+T)^m\cdot C(z)$ pour tout
  $m\in\mathbf{N}$ et tout $z\in D_{\rm rig}$. Par lin\'{e}arit\'{e}, on obtient
  $C(f(T)z)=f(T)C(z)$ pour tout $f\in L(T)$ et\footnote{Noter que si $z\in D_{\rm rig}$ et $P\in L[T]$ est non nul,
  alors $\frac{1}{P(T)}\cdot z\in D_{\rm rig}$.} tout $z\in D_{\rm rig}$. Or, il est facile
  de voir que $L(T)$
  est dense dans $\mathcal{R}$ (tout $f=\sum_{n\in\mathbf{Z}} a_n\cdot T^n\in \mathcal{R}$
  est la limite dans $\mathcal{R}$ de la suite $\sum_{|n|\leq N} a_n\cdot T^n$). On en d\'{e}duit que $C$ est $\mathcal{R}$-lin\'{e}aire,
  i.e. $C\in {\rm End}_{\varphi,\Gamma,\mathcal{R}}(D_{\rm rig})={\rm End}_{L[{\rm Gal}(\overline{\qp}/\qp)]}(V)$ (remarque
  \ref{pleinefidelite}). Comme $V$ est absolument irr\'{e}ductible,
  on conclut que $C$ est scalaire.

   La proposition \ref{HC} combin\'{e}e au lemme \ref{Casimir} montrent que
   $C-\frac{(a-b)^2-1}{2}$ envoie $D_{\rm rig}$ dans $t\cdot D_{\rm rig}$.
   Comme de plus $C-\frac{(a-b)^2-1}{2}$ est scalaire, d'apr\`{e}s la discussion pr\'{e}c\'{e}dente, on obtient finalement $C=\frac{(a-b)^2-1}{2}$. La formule pour
   $u^-$ suit alors du lemme \ref{Casimir}.

\end{proof}

 On finit par un corollaire concernant l'action
  de l'involution $w_D$ sur $D_{{\rm rig}}\boxtimes \zpet$.

\begin{corollary}
\label{strange}
 Pour tout $z\in D_{{\rm rig}}\boxtimes \zpet$ on a
 $$w_D(tz)=-\frac{P_{\rm Sen,V}(\nabla)(w_D(z))}{t}.$$
\end{corollary}

\begin{proof}
  Le th\'{e}or\`{e}me \ref{Casimirinf} montre que
 $$ w_D(tz)=w_D(u^+z)=wu^+z=u^-wz=
 u^-(w_D(z))=
 -\frac{P_{\rm Sen,V}(\nabla)(w_D(z))}{t}.$$

\end{proof}

\subsection{Application aux repr\'{e}sentations de Banach}

  Rappelons qu'une $L$-repr\'{e}sentation de Banach de
  ${\rm GL}_2(\qp)$ est un $L$-espace de Banach
  $\Pi$ muni d'une action continue de ${\rm GL}_2(\qp)$.
  La repr\'{e}sentation $\Pi$ est dite \textit{unitaire} si elle admet
  une norme d\'{e}finissant la topologie de Banach et qui est
  invariante sous l'action de ${\rm GL}_2(\qp)$. Elle est
  dite \textit{admissible} s'il existe une injection continue
  $H$-\'{e}quivariante $\Pi\subset \mathcal{C}(H,L)^n$ pour un
  sous-groupe ouvert compact $H$ de ${\rm GL}_2(\qp)$ et
  un $n\geq 1$. Cela \'{e}quivaut \`{a} ce que le dual topologique
  $\Pi^*$ de $\Pi$ soit un $L\otimes_{O_L} O_L[[{\rm GL}_2(\zp)]]$-module
  de type fini, o\`{u} $O_L[[{\rm GL}_2(\zp)]]$ est l'alg\`{e}bre de groupe compl\'{e}t\'{e}e\footnote{C'est la limite projective des alg\`{e}bres de groupe des quotients
  finis de $\rm{GL}_2(\zp)$. De mani\`{e}re plus conceptuelle, c'est l'alg\`{e}bre des mesures \`{a} valeurs dans $O_L$ sur ${\rm GL}_2(\zp)$.}
  de ${\rm GL}_2(\zp)$.

    Si $\Pi$ est une ${\rm GL}_2(\qp)$-repr\'{e}sentation
  de Banach admissible, on note $\Pi^{\rm an}$ le sous-espace
  des vecteurs localement analytiques de $\Pi$ (ce sont les vecteurs
  $v\in \Pi$ tels que $g\to g\cdot v$ soit une fonction localement analytique
  sur ${\rm GL}_2(\qp)$, \`{a} valeurs dans le Banach $\Pi$). D'apr\`{e}s un th\'{e}or\`{e}me
  fondamental de Schneider et Teitelbaum~(\cite{STInv}, theorem 7.1),
  il s'agit d'un sous-espace dense de $\Pi$. Le r\'{e}sultat suivant est d\^{u} \`{a} Colmez~\cite{Cbigone}.

  \begin{theorem}\label{correspondance}
   Soit $V$ une $L$-repr\'{e}sentation absolument irr\'{e}ductible, de dimension~$2$
   de ${\rm Gal}(\overline{\qp}/\qp)$ et soient $D$ et $D_{\rm rig}$ les
   $(\varphi,\Gamma)$-modules attach\'{e}s \`{a} $V$. Il existe une
   ${\rm GL}_2(\qp)$-repr\'{e}sentation de Banach unitaire admissible,
   topologiquement absolument irr\'{e}ductible $\Pi=\Pi(V)$ telle que l'on ait des suites
   exactes de ${\rm GL}_2(\qp)$-modules topologiques
    $$0\to \Pi^*\otimes (\delta_D\circ \det)\to D\boxtimes \p1\to \Pi\to 0,$$
    $$0\to (\Pi^{\rm an})^*\otimes (\delta_D\circ \det)\to D_{\rm rig}\boxtimes \p1\to \Pi^{\rm an}\to 0.$$
  \end{theorem}

  \begin{remark}\label{check}
   Si $V$, $D$ et $\Pi$ sont comme dans le th\'{e}or\`{e}me \ref{correspondance}, on note
   $\check{\Pi}=\Pi(V^*\otimes \chi)$, o\`{u} $V^*$ est le dual
   de $V$. Le choix d'un isomorphisme $\wedge^2(V)\simeq \chi\cdot\delta_D$ induit un isomorphisme $\check{\Pi}\simeq\Pi\otimes (\delta_D^{-1}\circ \det)$, qui, combin\'{e} aux suites exactes du th\'{e}or\`{e}me \ref{correspondance}, induit des inclusions ${\rm GL}_2(\qp)$-\'{e}quivariantes
  $\check{\Pi}^*\subset
   D\boxtimes \p1$ et $(\check{\Pi}^{\rm an})^*\subset D_{\rm rig}\boxtimes \p1$.

  \end{remark}

  Le r\'{e}sultat suivant r\'{e}pond \`{a} une question
  de Harris (~\cite{EmCoates}, remark 3.3.8).

  \begin{theorem}\label{carinf}
  Soit $V$ une $L$-repr\'{e}sentation
  absolument irr\'{e}ductible de $G_{\qp}$, de dimension $2$,
  \`{a} poids de Hodge-Tate g\'{e}n\'{e}ralis\'{e}s $a$ et $b$. Alors l'\'{e}l\'{e}ment de Casimir
  agit par multiplication par $\frac{(a-b)^2-1}{2}$ sur
  $\Pi(V)^{\rm an}$.
  \end{theorem}

   \begin{proof}
   D'apr\`{e}s le th\'{e}or\`{e}me \ref{Casimirinf}, l'\'{e}l\'{e}ment de Casimir
   agit par multiplication par $\frac{(a-b)^2-1}{2}$ sur
   $D_{\rm rig}\boxtimes\p1$. Comme $\Pi^{\rm an}$ est un quotient
   de $D_{\rm rig}\boxtimes \p1$ d'apr\`{e}s le th\'{e}or\`{e}me \ref{correspondance},
   le r\'{e}sultat suit.
   \end{proof}

   \begin{remark}
 Dans~\cite{DBenjamin}, on montre que toute repr\'{e}sentation localement
  analytique admissible et absolument irr\'{e}ductible d'un groupe de Lie $p$-adique
 admet un caract\`{e}re infinit\'{e}simal\footnote{ On montre en fait que les endomorphismes
 d'une telle repr\'{e}sentation sont tous scalaires.} et on conjecture que
 $\Pi^{\rm an}$ admet un caract\`{e}re infinit\'{e}simal pour toute repr\'{e}sentation
 de Banach admissible et absolument irr\'{e}ductible $\Pi$. Le probl\`{e}me est que
 $\Pi^{\rm an}$ n'est pas irr\'{e}ductible en g\'{e}n\'{e}ral et les m\'{e}thodes de~\cite{DBenjamin} (qui utilisent pleinement les r\'{e}sultats de~\cite{AW} et~\cite{STInv}) ne semblent pas s'adapter. Voici ce que l'on peut d\'{e}montrer avec les
 r\'{e}sultats de cet article: si $p>3$
 et si $\Pi$ est une ${\rm GL}_2(\qp)$-repr\'{e}sentation de Banach
    unitaire, admissible, absolument irr\'{e}ductible, alors
    $\Pi^{\rm an}$ admet un caract\`{e}re infinit\'{e}simal. En effet, si
    $\Pi$ est un sous-quotient d'une induite unitaire parabolique,
 le r\'{e}sultat est facile et sinon, d'apr\`{e}s un r\'{e}sultat de Pa\v{s}k\={u}nas~\cite{Pa}, conjectur\'{e}
 par Colmez, il existe une $L$-repr\'{e}sentation $V$ absolument irr\'{e}ductible, de dimension $2$
    de ${\rm Gal}(\overline{\qp}/\qp)$ telle que $\Pi\simeq \Pi(V)$ et donc le r\'{e}sultat suit du
    th\'{e}or\`{e}me \ref{carinf}.
 Si la r\'{e}duction modulo $p$ d'un r\'{e}seau invariant de $\Pi$ est
 de longueur finie, on peut d\'{e}montrer le m\^{e}me r\'{e}sultat sans l'hypoth\`{e}se
 $p>3$ et sans utiliser le th\'{e}or\`{e}me de Pa\v{s}k\={u}nas, mais la preuve reste bien d\'{e}tourn\'{e}e~\cite{{DfoncteurColmez}}.

 \end{remark}

\section{Mod\`{e}le de Kirillov et dualit\'{e}}\label{Kirillovetdualite}

  Ce chapitre est pr\'{e}liminaire \`{a} la preuve du th\'{e}or\`{e}me \ref{main2}.
  On commence par rappeler la construction d'un mod\`{e}le de Kirillov
   pour les vecteurs de $\Pi$ qui sont alg\'{e}briques sous l'action de l'unipotent sup\'{e}rieur. On rappelle ensuite le lien entre le mod\`{e}le de Kirillov
   et la dualit\'{e} entre $\Pi^{\rm an}$ et $(\Pi^{\rm an})^*$, qui jouera un
   r\^{o}le crucial dans le chapitre suivant.
  \textit{ Dans la suite on suppose que $V$ est absolument irr\'{e}ductible de dimension $2$, \`{a} poids de Hodge-Tate $0$ et
  $k\in\mathbf{N}^*$ (en particulier, $V$ est de Hodge-Tate).} Voir la d\'{e}finition \ref{phigamma} pour les
$(\varphi,\Gamma)$-modules $D,D_{\rm rig}, \tilde{D},\tilde{D}^+,\tilde{D}_{\rm dif}$ attach\'{e}s \`{a} $V$.
  Soit $\Pi$ la repr\'{e}sentation de Banach de ${\rm GL}_2(\qp)$ attach\'{e}e \`{a}
  $V$ par la correspondance de Langlands locale $p$-adique (th\'{e}or\`{e}me \ref{correspondance}).
Soit $U$ le radical unipotent du Borel sup\'{e}rieur de ${\rm GL}_2(\qp)$
  et soit $P$ le sous-groupe mirabolique de ${\rm GL}_2(\qp)$, form\'{e} des matrices
  du type $\left(\begin{smallmatrix} a & b \\0 & 1\end{smallmatrix}\right)$, avec
  $a\in\qpet$ et $b\in\qp$. Pour les notations $\check{\Pi}$ et
  $\check{\Pi}^{\rm an}$ voir la remarque \ref{check}.

\subsection{Le module $N_{{\rm dif},n}$}

  Le module $N_{{\rm dif},n}$ introduit dans cette partie va jouer un r\^{o}le crucial
  dans la suite. Il s'agit d'une version infinit\'{e}simale de l'\'{e}quation
  diff\'{e}rentielle attach\'{e}e par Berger~\cite{Ber} \`{a} une repr\'{e}sentation de de Rham.

\begin{proposition}\label{basedif}
  Posons $\varepsilon=1$ si $V$ n'est pas de Rham
    et $0$ dans le cas contraire. Il existe
    $e_1,e_2\in D_{{\rm dif},m(D)}^+[1/t]$ tels que

  a) $e_1,t^ke_2$ forment une base
     de $D_{{\rm dif},n}^+$ sur $L_n[[t]]$ pour tout $n\geq m(D)$.

  b) $\sigma_a(e_1)=e_1$ et $\sigma_a(e_2)=
      e_2+\varepsilon\log a\cdot e_1$ pour tout $a\in 1+p^{m(D)}\zp$.

\end{proposition}

\begin{proof}
C'est le contenu de la proposition VI.3.2 de~\cite{Cbigone}.
\end{proof}

\begin{definition} Les \'{e}l\'{e}ments $e_1,e_2$ de $D_{{\rm dif},n}^+$ \'{e}tant comme dans la proposition \ref{basedif}, soit
     $$N_{{\rm dif},n}=L_n[[t]]e_1\oplus L_n[[t]]e_2\subset t^{-k}D_{{\rm dif},n}^+.$$

\end{definition}

\begin{example}
 Si $V$ est de de Rham, on peut choisir pour $e_2$ une base de $D_{\rm dR}(V)/D^+_{\rm dR}(V)$, que l'on compl\`{e}te
 en une base $e_1,e_2$ de $D_{\rm dR}(V)$. On a alors
 $N_{{\rm dif},n}=L_n[[t]]\otimes_{L} D_{\rm dR}(V)$ pour tout
 $n$ suffisament grand.
\end{example}

  \subsection{Une caract\'{e}risation diff\'{e}rentielle des repr\'{e}sentations de de Rham}\label{eqdiff}

   Dans cette partie on donne un crit\`{e}re simple pour distinguer
   les repr\'{e}sentations de de Rham parmi les repr\'{e}sentations de Hodge-Tate.
   Ce genre d'argument est inspir\'{e} de~\cite{BC}, proposition 5.2.1. Notons
$\nabla_{2k}=\nabla(\nabla-1)...(\nabla-2k+1)$.

  \begin{lemma}
  \label{diff}
   Soit $n\geq m(D)$ et soient $e_1$ et $e_2$ comme dans la proposition \ref{basedif}.
  Notons $$X_n=\{z\in D_{{\rm dif},n}^+| \nabla_{2k}(z)\in t^{2k}N_{{\rm dif},n}\},$$

   a) Si $V$ est de de Rham, alors $X_n=D_{{\rm dif},n}^+$.

   b) Si $V$ est de Hodge-Tate mais pas de de Rham, alors $X_n=L_n[[t]]e_1+t^{2k}L_n[[t]]e_2$.
  \end{lemma}

\begin{proof} Notons que $\nabla_{2k}(t^j)=j(j-1)...(j-2k+1)t^j\in t^{2k}L_n[[t]]$
 pour tout $j\geq 0$. La partie a) suit alors de cette observation et du fait que
pour tous $A,B\in L_n[[t]]$ on a (noter que $\nabla e_1=\nabla e_2=0$)
 $$\nabla_{2k}(A(t)e_1+B(t)e_2)=\nabla_{2k}(A(t))e_1+\nabla_{2k}(B(t))e_2.$$ Supposons donc que $V$ n'est pas de de Rham,
 ce qui implique les \'{e}galit\'{e}s $\nabla(e_2)=e_1$ et $\nabla(e_1)=0$.
 D'apr\`{e}s la r\`{e}gle de Leibnitz, on a pour tout $B\in L_n[[t]]$ et tout $j\geq 0$
  $$\nabla^j(Be_2)=\nabla^j(B)e_2+j\nabla^{j-1}(B)e_1.$$ On en d\'{e}duit que pour tout
  $P\in L[[T]]$ on a $P(\nabla)(Be_2)=(P(\nabla)B)e_2+(P'(\nabla)B)e_1$, donc
   $$\nabla_{2k}(Be_2)=(\nabla_{2k}B)e_2+\left(\sum_{j=0}^{2k-1}\frac{\nabla_{2k}}{\nabla-j} B\right)e_1.$$
  Comme $\nabla_{2k}(L_n[[t]])\subset t^{2k}L_n[[t]]$, on obtient
  que $x=Ae_1+Be_2$ est dans $X_n$ si et seulement si $$\sum_{j=0}^{2k-1}\frac{\nabla_{2k}}{\nabla-j} B\in t^{2k}L_n[[t]].$$
  Mais si $B=\sum_{j\geq 0} \alpha_jt^j$, alors on a des \'{e}galit\'{e}s dans $L_n[[t]]/t^{2k}L_n[[t]]$
   $$ \sum_{j=0}^{2k-1}\frac{\nabla_{2k}}{\nabla-j} B=\sum_{s=0}^{2k-1}\alpha_s\cdot \sum_{j=0}^{2k-1}\frac{\nabla_{2k}}{\nabla-j}(t^s)$$
   $$=\sum_{s=0}^{2k-1} \alpha_s\cdot \frac{\nabla_{2k}}{\nabla-s}(t^s)=\sum_{s=0}^{2k-1} (-1)^{2k-1-s}s!\cdot (2k-s-1)!\alpha_s\cdot t^s.$$
  On a donc $x\in X_n$ si et seulement si $\alpha_s=0$ pour tout $0\leq s\leq 2k-1$, i.e. si et seulement si $B\in t^{2k}L_n[[t]]$. Cela
  permet de conclure.

\end{proof}

\subsection{Les accouplements  $[\,\,,\,]_{\p1}$ et $[\,\,,\,]_{\rm dif}$}\label{acdif} \label{acp1}

 Soit $V$ comme dans l'introduction de ce chapitre et identifions $\wedge^2(V)$ \`{a} $L$ par le choix d'une base de $\wedge^2(V)$.
  L'accouplement $V\times V\to L$ donn\'{e} par $(x,y)\to x\wedge y$ induit, par fonctorialit\'{e}, un accouplement
  $D\times D\to \mathcal{E}$, not\'{e} encore $(x,y)\to x\wedge y$, qui satisfait\footnote{Cette torsion par $\chi\cdot \delta_D$ vient du fait que
  l'accouplement $V\times V\to L$ n'est pas Galois-\'{e}quivariant, mais il le devient si l'on tord par $\chi\cdot\delta_D$.} (pour
  $x,y\in D$ et $a\in\zpet$)
    $$\sigma_a(x)\wedge \sigma_a(y)=a\delta_D(a)\cdot\sigma_a(x\wedge y), \quad \varphi(x)\wedge \varphi(y)=\delta_D(p)\cdot\varphi(x\wedge y).$$

  Notons ${\rm res}_0\left(f\frac{dT}{1+T}\right)$ le r\'{e}sidu en $0$ de la forme diff\'{e}rentielle $\frac{f}{1+T} dT$.
On d\'{e}finit un accouplement $[\, ,\,]: D\times
 D\to L$ par la formule
 $$[x,y]={\rm res}_0 \left((\sigma_{-1}(x)\wedge y)\frac{dT}{1+T}\right),$$
ainsi qu'un accouplement
  $[ \, \,, \, ]_{\p1}: (D\boxtimes \p1)\times
    (D\boxtimes \p1)\to L$ par la formule
  $$[(x_1,x_2),(y_1,y_2)]_{\p1}=[x_1,y_1]+[\varphi\psi(x_2),\varphi\psi(y_2)].$$ Sa restriction \`{a} $D\times D$ est l'accouplement
  $[\, , \,]$. Cette discussion s'applique aussi \`{a} $D_{\rm rig}$.

\begin{proposition}
\label{acc}
 1) $[\,\,,\,]_{\p1}$ est un accouplement parfait de
 $L$-espaces vectoriels topologiques, satisfaisant
 $[gx,gy]_{\p1}=\delta_D(\det g)[x,y]_{\p1}$ pour tous $x,y\in \Delta$
 et $g\in {\rm GL}_2(\qp)$.

 2) $\check{\Pi}^*$ (resp. $(\check{\Pi}^{\rm an})^*$) est son propre orthogonal dans $D\boxtimes \p1$
 (resp. $D_{{\rm rig}}\boxtimes \p1$).

\end{proposition}

\begin{proof}
  Toutes les r\'{e}f\'{e}rences sont \`{a}~\cite{Cbigone} (on a $[x,y]=\{x\otimes\delta_D^{-1},y\}$ avec les notations de loc.cit). 1) est la r\'{e}union des propositions I.2.1, I.2.2, II.1.11 (et des
   remarques faites avant la proposition V.2.10),
 et du th\'{e}or\`{e}me II.1.13. Ensuite, 2) est la r\'{e}union
 du th\'{e}or\`{e}me II.2.11, de la proposition
 V.2.10 et de la remarque V.2.21.
\end{proof}

  On d\'{e}duit du 1) de la proposition pr\'{e}c\`{e}dente que pour tout
  $k\geq 1$ et $x,y\in D_{\rm rig}$ on a
$$[(u^-)^kx, y]_{\p1}=
(-1)^k[x, (u^-)^k y]_{\p1}.$$

   L'accouplement $V\times V\to L$ induit aussi un accouplement $\tilde{D}_{\rm dif}\times \tilde{D}_{\rm dif}\to L\otimes_{\qp} \mathbf{B}^H_{\rm dR}$,
   not\'{e} $(x,y)\to x\wedge y$, satisfaisant $\sigma_a(x)\wedge \sigma_a(y)=a\delta_D(a)\cdot\sigma_a(x\wedge y)$ pour $x,y\in \tilde{D}_{\rm dif}$
   et $a\in\zpet$. Soit ${\rm Res}_{\zp} : L\otimes_{\qp} \mathbf{B}^H_{\rm dR} \to L((t))$ la trace de Tate normalis\'{e}e (voir \ref{Tate}) et,
si $f=\sum_{n>>-\infty} a_nt^n\in L((t))$, notons ${\rm res}_0(f dt)=a_{-1}$.
On d\'{e}finit un accouplement
  $[\, ,\,]_{\rm dif}: \tilde{D}_{\rm dif}\times
   \tilde{D}_{\rm dif}\to L$ par $$[x,y]_{{\rm dif}}={\rm res}_0 \left({\rm{Res}}_{\zp} (\sigma_{-1}(x)\wedge
 y)dt\right).$$

 Le r\'{e}sultat suivant se d\'{e}duit\footnote{C'est une simple traduction.} de~\cite[lemme VI.4.16]{Cbigone}.

 \begin{lemma}
 \label{orth}
  L'orthogonal dans $D_{{\rm dif},n}^+[1/t]$ de $N_{{\rm dif},n}$ est
  $t^kN_{{\rm dif}, n}$ et $D_{{\rm dif},n}^+$ est son propre orthogonal
  dans $D_{{\rm dif},n}^+[1/t]$.
 \end{lemma}

  La trace de Tate $\frac{1}{p}{\rm Tr}_{L_{n+1}/L_n}$
  envoie $D_{\rm dif,n+1}^+/t^kN_{\rm dif,n+1}$ dans
  $D_{{\rm dif},n}^+/t^kN_{{\rm dif},n}$. Si
  $x=(x_n)_n\in \lim\limits_{{\longleftarrow}} D_{{\rm dif},n}^+/t^kN_{{\rm dif},n}$ (la limite est relativement
  aux traces de Tate) et si $y\in X_{\infty}^-= \lim\limits_{{\longrightarrow}} N_{{\rm dif},n}/
  D_{{\rm dif},n}^+$, on d\'{e}finit $[x,y]_{\rm dif}$ de la fa\c{c}on suivante:
  on choisit $n$ tel que $y\in N_{{\rm dif},n}/D_{{\rm dif},n}^+$, ainsi que des rel\`{e}vements
  $\hat{y}_n\in N_{{\rm dif},n}$ de $y$ et $\hat{x}_n\in D_{{\rm dif},n}^+$ de $x_n$ et on pose
  $[x,y]_{\rm dif}=[\hat{x}_n,\hat{y}_n]_{\rm dif}$.

\begin{lemma}
\label{acclimite}
  Cet accouplement est bien d\'{e}fini et non d\'{e}g\'{e}n\'{e}r\'{e}.
\end{lemma}

\begin{proof}
 Le fait que $[x,y]_{\rm dif}$ ne d\'{e}pend pas des choix de
 $\hat{x}_n$ ou $\hat{y}_n$ d\'{e}coule du lemme
 \ref{orth}. Pour montrer qu'il ne d\'{e}pend pas de $n$, il suffit de voir
 que si $\hat{y}_{n+1}-\hat{y}_n\in D_{\rm dif,n+1}^+$
 et $\hat{x}_{n+1}, \hat{x}_n$ s'envoient sur
 $x_{n+1}, x_n$, alors $[\hat{x}_{n+1}, \hat{y}_{n+1}]_{\rm dif}=
 [\hat{x}_n,\hat{y}_n]_{\rm dif}$. Or, $\hat{y}_{n+1}-\hat{y}_n$
 \'{e}tant dans $D_{\rm dif,n+1}^+$, il est orthogonal \`{a}
 $\hat{x}_{n+1}$, donc il reste \`{a} voir que
 $\hat{y}_n$ est orthogonal \`{a} $\hat{x}_{n+1}-\hat{x}_n$.
 Comme $p^{-1}{\rm Tr}_{L_{n+1}/L_n}(x_{n+1})=x_n$, il reste
 en fait \`{a} voir que si $y\in N_{{\rm dif},n}$ et
 $x\in D_{\rm dif,n+1}^+$, alors $p^{-1}{\rm Tr}_{L_{n+1}/L_n}(x)-x$
 est orthogonal \`{a} $y$. Cela est imm\'{e}diat sur la formule d\'{e}finissant
 l'accouplement. Le fait que l'accouplement est non d\'{e}g\'{e}n\'{e}r\'{e} suit du
 lemme \ref{orth}.

\end{proof}

\subsection{La restriction de $\Pi$ au mirabolique}

Rappelons que $H={\rm Gal}(\overline{\qp}/ {\qp}(\mu_{p^{\infty}}))={\rm Ker}\chi$.
Si $b\in\qp$, soit $n$ tel que $p^nb\in\zp$ et posons $$[(1+T)^b]=\varphi^{-n}\left((1+T)^{p^nb}\right)=\varphi^{-n}\left(\sum_{k\geq 0} \binom{p^nb}{k}T^k\right).$$
On obtient ainsi un \'{e}l\'{e}ment de $(\mathbf{\tilde{B}}^+)^H$, qui ne d\'{e}pend pas du choix
de $n$. L'application $b\to [(1+T)^b]$ est un analogue $p$-adique de $x\to e^{2i\pi x}$ et la transform\'{e}e de Fourier $\mu\to \int_{\qp} [(1+T)^x]\mu$ identifie $L\otimes_{\qp}(\mathbf{\tilde{B}}^+)^H$
\`{a} l'espace des mesures sur $\qp$, nulles
\`{a} l'infini (voir la proposition IV.3.1 de~\cite{Cmirab}).
Comme $\varphi$ est inversible sur $\tilde{D}^+=(\mathbf{\tilde{B}}^+\otimes_{\qp} V)^H$ et $\tilde{D}=(\mathbf{\tilde{B}}\otimes_{\qp} V)^H$, on peut munir ces deux modules
 d'une action du mirabolique $P$ gr\^{a}ce \`{a} (si $k\in\mathbf{Z}$, $a\in\zpet$, $b\in\qp$)
   $$\left(\begin{smallmatrix} p^ka & b \\0 & 1\end{smallmatrix}\right) \tilde{z}=[(1+T)^{b}]\varphi^{k}(\sigma_a(\tilde{z})).$$
D'apr\`{e}s~\cite[lemme IV.1.2]{Cmirab}, tout \'{e}l\'{e}ment $z$ de
$\tilde{D}$ s'\'{e}crit de mani\`{e}re unique sous la forme $z=\sum_{i\in I} [(1+T)^i]z_i$, o\`{u} $I\subset \qp$ est un
syst\`{e}me de repr\'{e}sentants de $\qp/\zp$ et $z_i\in D$ tend vers $0$ quand $i\to \infty$ dans $\qp$. Soit
$I_n=I\cap p^{-n}\zp$ et notons
 $$x_n(z)=\sum_{i\in I_n} \left(\begin{smallmatrix} 1 & i \\0 & 1\end{smallmatrix}\right) z_i\in
D\boxtimes p^{-n}\zp\subset D\boxtimes \p1.$$ D'apr\`{e}s le lemme II.1.16 de~\cite{Cbigone}, la suite
$(x_n(z))_{n\geq 0}$ d'\'{e}l\'{e}ments de $D\boxtimes \p1$ converge vers un \'{e}l\'{e}ment
que l'on note encore $z\in D\boxtimes \p1$ et l'application $\tilde{D}\to D\boxtimes\p1$ ainsi obtenue
est une injection $P$-\'{e}quivariante. Le corollaire II.2.8 de~\cite{Cbigone} montre que cette inclusion
induit une inclusion $\tilde{D}^+\subset \check{\Pi}^*$ (en fait, le lemme cit\'{e} montre que l'on a m\^{e}me
$\check{\Pi}^*=\tilde{D}^++w\cdot\tilde{D}^+$). Le r\'{e}sultat suivant~(\cite{Cbigone}, cor.
    II.2.9) sera tr\`{e}s utile pour la suite.

  \begin{proposition}
  \label{mirab}
   L'inclusion de $\tilde{D}$ dans $D\boxtimes \p1$ induit un
 isomorphisme de $P$-modules de Banach $\tilde{D}/\tilde{D}^+\simeq \Pi$.
  \end{proposition}

\subsection{Le mod\`{e}le de Kirillov de Colmez}\label{KirCol}

\begin{definition}\label{UPalg}
\label{UPalg}
a)  $v\in\Pi$ est dit $U$-alg\'{e}brique s'il existe $n$ tel que
$\left(\left(\begin{smallmatrix} 1 & p^n \\0 & 1\end{smallmatrix}\right)-1\right)^kv=0$, ce qui \'{e}quivaut \`{a} ce que
$x\to \left(\begin{smallmatrix} 1 & x \\0 & 1\end{smallmatrix}\right) v$ soit localement polynomiale de degr\'{e} plus petit que $k$.
On note $\Pi^{U-{\rm alg}}$ l'espace des vecteurs $U$-alg\'{e}briques.

b)  $v\in\Pi$ est dit $P$-alg\'{e}brique si $v$ est $U$-alg\'{e}brique et s'il existe $n$ tel que $\lambda_n(v)=0$, o\`{u}
$$\lambda_n=\prod_{i=0}^{k-1}\left(\left(\begin{smallmatrix} 1+p^n & 0 \\0 & 1\end{smallmatrix}\right)-(1+p^n)^i \right)\in O_L [{\rm GL}_2(\zp)].$$
Cela \'{e}quivaut \`{a} ce que les applications $x\to \left(\begin{smallmatrix} 1 & x \\0 & 1\end{smallmatrix}\right) v$ et $x\to \left(\begin{smallmatrix} x & 0 \\0 & 1\end{smallmatrix}\right)v$ soient localement polynomiales de degr\'{e} plus petit que $k$. On note $\Pi^{P-{\rm alg}}$ l'espace des vecteurs  $P$-alg\'{e}briques.

\end{definition}

  Soit $v\in \Pi$ et soit\footnote{Rappelons que d'apr\`{e}s la proposition \ref{mirab}, la repr\'{e}sentation $\Pi$ restreinte \`{a} $P$ est isomorphe \`{a}
$\tilde{D}/\tilde{D}^+$.} $\tilde{z}\in \tilde{D}$ un rel\`{e}vement de $v$.
Comme $\left(\begin{smallmatrix} 1 & b \\0 & 1\end{smallmatrix}\right)$
  agit sur $\tilde{D}$ par multiplication par $(1+T)^b$ pour $b\in\zp$,
  on a $v\in\Pi^{U-\rm alg}$ si et seulement si
  $\tilde{z}\in \frac{1}{\varphi^n(T)^k}\cdot\tilde{D}^+$
  pour un certain $n$. Puisque $\varphi^n(T)$ divise $t$ dans $\mathbf{B}_{\rm dR}^+$
et $\mathbf{\tilde{B}}^+\subset \mathbf{B}_{\rm dR}^+$, on obtient
  $\tilde{z}\in t^{-k}\tilde{D}_{\rm dif}^+$.

  Comme $\left(\begin{smallmatrix} x & 0 \\0 & 1\end{smallmatrix}\right) \tilde{z}\pmod {\tilde{D}^+}$ reste
 $U$-alg\'{e}brique pour tout $x\in \qpet$, cela nous fournit une application
   $$\phi_v: \qpet\to t^{-k} \tilde{D}^+_{\rm dif}/\tilde{D}^+_{\rm dif},\quad
   \phi_v(x)=\left(\begin{smallmatrix} x & 0 \\0 & 1\end{smallmatrix}\right) \tilde{z}\pmod{ \tilde{D}_{\rm dif}^+}.$$
   Le r\'{e}sultat suivant en r\'{e}sume les propri\'{e}t\'{e}s essentielles.

  \begin{proposition}
  \label{Kir}
   Soit $v\in\Pi^{U-{\rm alg}}$.

  1) $\phi_v$ est bien d\'{e}finie (i.e. ne d\'{e}pend pas du choix de $\tilde{z}$), \`{a} support compact dans $\qp$
 et $v\to\phi_v$ est injective.

  2) On a $\sigma_a(\phi_v(x))=\phi_v(ax)$ pour $a\in \zpet$
  et $x\in \qpet$. De plus, si $g=\left(\begin{smallmatrix} a & b \\0 & d\end{smallmatrix}\right)$, on a
  $$\phi_{gv}(x)=\delta_D(d)[(1+T)^{\frac{bx}{d}}]\phi_v\left(\frac{ax}{d}\right).$$

  3) $\Pi^{P-{\rm alg}}$ est stable par le Borel sup\'{e}rieur et
$$\Pi^{P-{\rm alg}}=\{v\in \Pi^{U-{\rm alg}}|\quad {\rm Im}(\phi_v)\subset X_{\infty}^-\},$$ o\`{u} $X_{\infty}^-$ est la r\'{e}union croissante des sous-espaces
$N_{{\rm dif},n}/D_{{\rm dif},n}^+$ de $t^{-k}\tilde{D}_{\rm \rm dif}^+/\tilde{D}_{\rm dif}^+$.
  \end{proposition}

\begin{proof}
  Voir les lemmes VI.5.4
  et VI.5.5 de~\cite{Cbigone}.

\end{proof}

  \begin{definition}
  \label{Picp}
   Soit $\Pi_{c}^{P-{\rm alg}}$ l'espace des vecteurs $v\in \Pi^{P-{\rm alg}}$
   tels que $\phi_v$ soit \`{a} support compact dans $\qpet$
   (i.e. $\phi_v$ est nulle sur $p^n\zp-\{0\}$ pour tout $n$ assez grand).
  \end{definition}

\subsection{Dualit\'{e} et mod\`{e}le de Kirillov}

   Le corollaire II.7.2 et le lemme V.2.15 de~\cite{Cbigone}
  montrent que  ${\rm{Res}}_{\zp} ((\check{\Pi}^{\rm an})^*)\subset
  D^{]0,r_{m(D)}]}$. On peut donc d\'{e}finir pour $n\geq m(D)$, $l\in\mathbf{Z}$ et $z\in (\check{\Pi}^{\rm an})^*$
 la quantit\'{e}
 $$i_{l,n}(z)=\varphi^{-n}\left( {\rm Res}_{\zp}\left( \left(\begin{smallmatrix} p^{n-l} & 0 \\0 & 1\end{smallmatrix}\right)z\right)\right)
 \in D_{{\rm dif},n}^+.$$
 Pour tout $z\in D^{]0,r_{n+1}]}$ et $n\geq m(D)$, on a $$\varphi^{-n}(\psi(z))=p^{-1}{\rm Tr}_{L_{n+1}/L_n}\varphi^{-(n+1)}(z).$$
 De plus, pour tout $x\in D_{{\rm rig}}\boxtimes \p1$ on a
  $${\rm Res}_{\zp}\left(\left(\begin{smallmatrix} p^{-1} & 0 \\0 & 1\end{smallmatrix}\right)x\right)=\psi({\rm Res}_{\zp}(x)).$$
On en d\'{e}duit que
  $\frac{1}{p}{\rm Tr}_{L_{n+1}/L_n}(i_{l,n+1}(z))=i_{l,n}(z)$ pour
  $n\geq m(D)$, $l\in\Z$ et $z\in (\check{\Pi}^{\rm an})^*$.
  On dispose donc d'un \'{e}l\'{e}ment\footnote{Les applications de transition dans $ \lim\limits_{{\longleftarrow}} D_{{\rm dif},n}^+/t^kN_{{\rm dif},n}$
  sont les traces de Tate normalis\'{e}es.}
  $$i_l^+(z)=(i_{l,n}(z)\pmod {t^kN_{{\rm dif},n}})_{n\geq m(D)}\in \lim\limits_{{\longleftarrow}} D_{{\rm dif},n}^+/t^kN_{{\rm dif},n}.$$
 Cette observation et le 3) de la proposition \ref{Kir}  donnent un sens \`{a} l'\'{e}nonc\'{e} suivant
 (pour la d\'{e}finition des accouplements $[\,\,,\,]_{\p1}$ et $[\,\,,\,]_{\rm dif}$ voir \ref{acp1}).

  \begin{proposition}
  \label{deep}
    $\Pi_c^{P-{\rm alg}}$ est contenu dans $\Pi^{\rm an}$ et pour tout
    $z\in (\check{\Pi}^{\rm an})^*$ et $v\in \Pi_c^{P-{\rm alg}}$ on a
    $$[z,v]_{\p1}=\sum_{l\in\mathbf{Z}}\delta_D(p^l)[i_l^+(z), \phi_{v}(p^{-l})]_{\rm dif}.$$
  \end{proposition}

\begin{proof}
 C'est le contenu de la proposition VI.5.12 de~\cite{Cbigone}, dont
 l'ingr\'{e}dient de base est la section VI.1 de~\cite{Cbigone} (elle est donc ind\'{e}pendante
 du reste du chapitre VI).
\end{proof}

\section{Vecteurs localement alg\'{e}briques et repr\'{e}sentations de de Rham}\label{deRham}

  Le but de ce chapitre est de d\'{e}montrer le th\'{e}or\`{e}me \ref{main2}. L'ingr\'{e}dient principal est l'\'{e}tude de $\Pi_c^{P-{\rm alg}}$, \'{e}tude
qui se fait d'une part gr\^{a}ce \`{a} la th\'{e}orie de Hodge $p$-adique (proposition
\ref{picpgros}), de l'autre en reliant le mod\`{e}le de Kirillov et la dualit\'{e} entre
 $\Pi^{\rm an}$ et $(\Pi^{\rm an})^*$ (proposition \ref{deep}). Une fois ces r\'{e}sultats
 \'{e}tablis, le th\'{e}or\`{e}me \ref{main2} est une cons\'{e}quence plus ou moins formelle
 du th\'{e}or\`{e}me \ref{main1}. Dans ce chapitre $V$ est une $L$-repr\'{e}sentation
 absolument irr\'{e}ductible, de dimension $2$ de ${\rm Gal}(\overline{\qp}/\qp)$,
 $D$ et $D_{\rm rig}$ sont les $(\varphi,\Gamma)$-modules attach\'{e}s \`{a} $V$
 et $\Pi$ est la ${\rm GL}_2(\qp)$-repr\'{e}sentation de Banach attach\'{e}e \`{a}
 $V$ (th\'{e}or\`{e}me \ref{correspondance}).

\subsection{Sorites sur les vecteurs localement alg\'{e}briques}

Soit $G$ un $\qp$-groupe alg\'{e}brique r\'{e}ductif, que l'on identifie
   \`{a} ses $\qp$-points et soit $\pi$ une $L$-repr\'{e}sentation
   localement analytique de $G$ (au sens de~\cite{STJAMS}).
    Soit $\mathfrak{g}={\rm Lie}(G)$ et ${\rm Rep}^{{\rm alg}}_L(G)$ la cat\'{e}gorie des $L$-repr\'{e}sentations du groupe alg\'{e}brique $G$.

 \begin{definition}
\label{localg}
 a) Si $W\in {\rm Rep}^{{\rm alg}}_L(G)$, soit $\pi_{W-{\rm lalg}}$ l'espace des vecteurs
 $v\in\pi$ qui sont dans l'image d'un morphisme $H$-\'{e}quivariant
 $f: W^n\to \pi$, pour un sous-groupe ouvert $H$ de $G$ et un entier $n\geq 1$.
 C'est un sous-espace de $\pi$ stable par $G$.

 b) On dit que $v\in \pi$ est localement alg\'{e}brique et on \'{e}crit
 $v\in \pi^{{\rm alg}}$ s'il existe $W\in {\rm Rep}^{{\rm alg}}_L(G)$ telle que
 $v\in \pi_{W-{\rm lalg}}$. On dit que $\pi$ est localement alg\'{e}brique
 si $\pi^{{\rm alg}}=\pi$.

 \end{definition}

  Cette d\'{e}finition, tir\'{e}e de~\cite{Emlocan} est \'{e}quivalente aux celles
  utilis\'{e}es par Schneider-Teitelbaum~\cite{STUfinite} ou Colmez~\cite{Cbigone}, gr\^{a}ce \`{a} la proposition
  4.2.8 de~\cite{Emlocan}. Soit ${\rm Hom}(W,\pi)^\mathfrak{g}$ l'espace des morphismes
   $\mathfrak{g}$-\'{e}quivariants de $W$ dans $\pi$.

   \begin{proposition}
\label{sorites}
   a) Soit $W\in {\rm Rep}^{{\rm alg}}_{L}(G)$. Le morphisme
   naturel $${\rm Hom}(W,\pi)^{\mathfrak{g}}\otimes_{L} W\to \pi_{W-{\rm lalg}}$$ est un isomorphisme topologique
   et $\pi_{W-l{\rm alg}}$ est un sous-espace ferm\'{e} de $\pi$.

   b) Soit $\hat{G}$ un syst\`{e}me de repr\'{e}sentants des objets irr\'{e}ductibles de
   ${\rm Rep}^{{\rm alg}}_L(G)$. L'application naturelle
   $$\oplus_{W\in \hat{G}} \pi_{W-{\rm lalg}}\to \pi^{{\rm alg}}$$ est un isomorphisme.
   \end{proposition}

   \begin{proof}
    Voir la section 4.2 de~\cite{Emlocan}.
   \end{proof}

    Le r\'{e}sultat suivant est d\^{u} \`{a} Colmez~\cite[prop. VI.5.1]{Cbigone}. Nous en donnons une d\'{e}monstration plus directe, bas\'{e}e sur le th\'{e}or\`{e}me \ref{carinf}. On note
    ${\rm Sym}^{k-1}(L^2)$ la puissance sym\'{e}trique $(k-1)$-i\`{e}me de la repr\'{e}sentation
    standard de $\rm{GL}_2(\qp)$ sur $L\oplus L$.

\begin{proposition}
\label{HTdif}
a) Si $\Pi^{{\rm alg}}\ne 0$, alors $V$ est
 Hodge-Tate \`{a} poids distincts.

b) Si les poids de Hodge-Tate de $V$ sont $0$ et
$k\in\mathbf{N}^*$ et si $\Pi(V)^{{\rm alg}}\ne 0$, alors
 $$\Pi^{{\rm alg}}={\rm Sym}^{k-1}(L^2)\otimes_{L}\Pi^{\rm lc}$$
 pour une repr\'{e}sentation lisse admissible $\Pi^{\rm lc}$.

\end{proposition}

\begin{proof}

 Si $k\geq 1$ et $l$ sont des entiers, soit $W_{l,k}={\rm Sym}^{k-1} L^2\otimes
 \det^l$. Supposons que $W_{l,k}\otimes \Pi^{\rm lc}$
 est une sous-repr\'{e}sentation de $\Pi$, pour une repr\'{e}sentation lisse
 $\Pi^{\rm lc}$. Soient $a$ et $b$ les poids de Hodge-Tate de $V$.
 En consid\'{e}rant les caract\`{e}res centraux, on obtient $a+b=k+2l$.
 L'examen des caract\`{e}res infinit\'{e}simaux (en utilisant
 le th\'{e}or\`{e}me \ref{carinf}) fournit $k^2=(a-b)^2$. Comme
 $k$ est un entier strictement positif, cela montre que
 $a$ et $b$ sont des entiers distincts et que $k$ et
 $l$ sont uniquement d\'{e}t\'{e}rmin\'{e}s par $a$ et $b$.
 Ceci et la proposition \ref{sorites} permettent de conclure
 (l'admissibilit\'{e} de $\Pi^{\rm lc}$ suit de celle de $\Pi$, car $\Pi^{\rm lc}={\rm Hom}({\rm Sym}^{k-1}(L^2),\Pi^{\rm an})^{\mathfrak{g}}$ est un
 sous-espace ferm\'{e} de ${\rm Hom}({\rm Sym}^{k-1}(L^2),\Pi^{\rm an})$, qui est une repr\'{e}sentation localement analytique admissible puisque
 ${\rm Sym}^{k-1}(L^2)$ est de dimension finie; on conclut alors par la proposition 6.4 et le th\'{e}or\`{e}me 6.6 de \cite{STInv}).
\end{proof}

  \subsection{L'espace $\Pi^{P-{\rm alg}}_c$
  et vecteurs presque alg\'{e}briques}

La proposition
   suivante est due \`{a} Colmez~\cite[lemme VI.5.10]{Cbigone}.
   On en donne une autre d\'{e}monstration. L'espace $\Pi_{c}^{P-\rm {alg}}$ a \'{e}t\'{e} d\'{e}fini dans
   \ref{Picp}.

   \begin{proposition}
   \label{picpgros}
     L'application $v\to (\phi_v(p^i))_{i\in\mathbf{Z}}$ induit une bijection
     de $\Pi_c^{P-{\rm alg}}$ sur $\oplus_{i\in\mathbf{Z}} X_{\infty}^-$
     \rm (voir la proposition \ref{Kir}, 3) pour $X_{\infty}^-$).
   \end{proposition}

\begin{proof}
  L'image est bien contenue dans $\oplus_{i\in\mathbf{Z}} X_{\infty}^-$,
  car ${\rm Im}(\phi_v)\subset X_{\infty}^-$  (prop.\ref{Kir}) et $\phi_v(p^i)=0$ pour
  $|i|>>0$ (par d\'{e}finition). Si $\phi_v(p^i)=0$ pour tout $i$, on a aussi
$\phi_v(x)=0$ pour tout $x$ d'apr\`{e}s le 2) de la prop. \ref{Kir} et
l'injectivit\'{e} de $v\to\phi_v$ montre que $v=0$. La partie d\'{e}licate est la surjectivit\'{e}.
Pour cela, notons d'abord qu'il suffit de montrer que pour tout
  $x\in X_{\infty}^-$, on peut trouver $v\in \Pi_c^{P-{\rm alg}}$ tel
  que $\phi_v(p^i)=1_{i=0}x$.  En effet, si
  ceci est vrai et si $(x_i)_i$ est une suite presque nulle dans
  $X_{\infty}^-$, il existe pour chaque $l\in\Z$ un $v_l'\in\Pi_{c}^{P-{\rm alg}}$ tel que
  $\phi_{v_l'}(p^i)=1_{i=0}x_{l}$. L'injectivit\'{e} de $v\to \phi_v$ et le
  fait que $(x_l)_l$ est une suite presque nulle
  montrent que $(v_l')_l$ est une suite presque nulle.  On peut donc d\'{e}finir
  $$v=\sum_{l\in\Z} \left(\begin{smallmatrix} p^{-l} & 0 \\0 & 1\end{smallmatrix}\right)
  v_l'$$ et on v\'{e}rifie facilement que $\phi_v(p^i)=x_i$.

   Soit donc $x\in X_{\infty}^-$ et soient $n\geq 1$, $\hat{x}\in N_{{\rm dif},n}$
   dont l'image dans $X_{\infty}^-$ est $x$. Soit $\omega=\frac{T}{\varphi^{-1}(T)}\in
   \mathbf{\tilde{A}}^+$. On va chercher $v$ de la forme $v=\tilde{z}\pmod {\tilde{D}^+}\in \Pi=\tilde{D}/\tilde{D}^+$
   avec $\tilde{z}=\omega^{-k} \tilde{y}$ et  $ \tilde{y}\in \tilde{D}^+$.  Comme
   $\theta(\varphi^i(\omega))\ne 0$ pour $i\ne 0$,
   pour un tel $\tilde{z}$ on a forc\'{e}ment $\varphi^i(\tilde{z})\in \tilde{D}_{\rm dif}^+$
   pour tout $i\ne 0$. Ceci combin\'{e} au fait que $\phi_v(ax)=\sigma_a(\phi_v(x))$
   pour $a\in\zpet$ montre que ${\rm supp}({\phi_v})\subset \zpet$.
   La condition $\phi_v(1)=x$ \'{e}quivaut \`{a} $\tilde{y}-\omega^k \hat{x}\in \omega^k \tilde{D}_{\rm dif}^+$.
   Comme $\omega^k \mathbf{B}_{\rm dR}^+= t^k\mathbf{B}_{\rm dR}^+$ et $t^kN_{{\rm dif},n}\subset D_{{\rm dif},n}^+\subset
   \tilde{D}_{\rm dif}^+$, l'\'{e}l\'{e}ment $u=\omega^k \hat{x}$ est dans $\tilde{D}^+_{\rm dif}$.
   On peut donc conclure quant \`{a} l'existence de $\tilde{y}$ gr\^{a}ce au lemme suivant:

    \begin{lemma}
   \label{useful}
    Pour tous $N\geq 1$ et $u\in \tilde{D}_{\rm dif}^+$, il existe $\tilde{y}\in\tilde{D}^+$
    tel que $\tilde{y}-u\in\omega^N \tilde{D}^{+}_{\rm dif}$.
   \end{lemma}

   \begin{proof} On d\'{e}montre le r\'{e}sultat par r\'{e}currence sur $N$.
   D'apr\`{e}s~\cite[lemme V.1.7]{Cbigone}, l'application $\theta: \tilde{D}^+\to (\mathbf{C}_p\otimes_{\qp} V)^H $ est
    surjective. Donc pour tout $u\in \tilde{D}_{\rm dif}^+$ il existe $\tilde{y}\in\tilde{D}^+$ tel que
   $\theta(\tilde{y})=\theta(u)$, i.e. $\tilde{y}-u\in \omega\cdot \tilde{D}^{+}_{\rm dif}$. Cela d\'{e}montre le cas $N=1$. Supposons le r\'{e}sultat
   vrai pour $N$ et soit $u\in \tilde{D}_{\rm dif}^+$. Il existe $\tilde{y}\in\tilde{D}^+$ et $u_1\in \tilde{D}_{\rm dif}^+$ tel que
   $\tilde{y}-u=\omega^N\cdot u_1$. Soit $\tilde{y_1}\in\tilde{D}^+$ tel que $\tilde{y_1}-u_1\in \omega\cdot \tilde{D}_{\rm dif}^+$.
   Alors $(\tilde{y}-\omega^N\cdot \tilde{y_1})-u\in \omega^{N+1}\cdot \tilde{D}_{\rm dif}^+$.

   \end{proof}

   Pour finir la preuve, il reste \`{a} v\'{e}rifier que $v\in\Pi_{c}^{P-{\rm alg}}$. Par construction,
   $T^k\tilde{z}\in \tilde{D}^+$, donc $v$ est tu\'{e} par $\left( \left(\begin{smallmatrix} 1 & 1 \\0 & 1\end{smallmatrix}\right)-1\right)^k$,
   donc $v\in\Pi^{U-{\rm alg}}$. De plus, comme $$\prod_{j=0}^{k-1} (\sigma_{1+p^n}-(1+p^n)^i)N_{{\rm dif},n}\subset
   D_{{\rm dif},n}^+$$ (cela est imm\'{e}diat, voir par exemple~\cite[lemme VI.4.1]{Cbigone}) et comme
    ${\rm Im}(\phi_v)\subset N_{{\rm dif},n}/D_{{\rm dif},n}^+$, on a $\phi_{\lambda_n v}=0$
    (cf. d\'{e}finition \ref{UPalg} pour $\lambda_n$) et donc $\lambda_n v=0$ et $v\in \Pi^{P-{\rm alg}}$.
     Par construction, ${\rm Im}(\phi_v)$ est compact
   dans $\qpet$, d'o\`{u} le r\'{e}sultat.

\end{proof}

\subsection{Vecteurs presque alg\'{e}briques et repr\'{e}sentations de de Rham}

  On d\'{e}montre dans cette partie le th\'{e}or\`{e}me \ref{main2}. Cela utilise essentiellement tous les r\'{e}sultats
  des chapitres pr\'{e}c\'{e}dents. On suppose que $V$ est \`{a} poids de Hodge-Tate
  $0$ et $k\in\mathbf{N}^*$ (d'apr\`{e}s la proposition \ref{HTdif}, \`{a} un twist pr\`{e}s ces hypoth\`{e}ses sont n\'{e}cessaires si on esp\`{e}re trouver des vecteurs alg\'{e}briques).
 Soit $\Pi_{c}^{P-\rm alg}$ l'espace d\'{e}fini dans
   \ref{Picp}. On commence par d\'{e}montrer la caract\'{e}risation suivante des
  repr\'{e}sentations de de Rham.

  \begin{theorem}\label{crucialalg}
   $V$ est de de Rham si et seulement si $(u^-)^k$ tue $\Pi_{c}^{P-\rm alg}$.
  \end{theorem}

  \begin{proof}
   Par dualit\'{e}, $(u^-)^k$ tue $\Pi_{c}^{P-\rm alg}$ si et seulement si
   $(u^-)^k(\check{\Pi}^{\rm an})^*$ est orthogonal \`{a} $\Pi_{c}^{P-\rm alg}$ pour
   l'accouplement $[\,\,,\,]_{\p1}$ (voir \ref{acp1}).
  La proposition \ref{deep} combin\'{e}e \`{a} la proposition
   \ref{picpgros} et au lemme \ref{acclimite} montre que ceci se passe si et seulement si $i_l^+((u^-)^kz)=0$ dans $\lim\limits_{{\longleftarrow}} D_{{\rm dif},n}^+/t^kN_{{\rm dif},n}$,
   pour tout $l\in\mathbf{Z}$ et pour tout $z\in (\check{\Pi}^{\rm an})^*$, ou encore
   $i_{l,n}((u^-)^kz)\in t^kN_{{\rm dif},n}$ pour tout
   $n\geq m(D)$, $l\in\mathbf{Z}$ et $z\in (\check{\Pi}^{\rm an})^*$. Par d\'{e}finition des $i_{l,n}$, cela arrive
   si et seulement si $\varphi^{-n}\left((u^{-})^k {\rm Res}_{\zp}(z) \right)\in  t^kN_{{\rm dif},n}$ pour tout $n\geq m(D)$ et tout
    $z\in (\check{\Pi}^{\rm an})^*$. Nous aurons besoin du r\'{e}sultat suivant:

    \begin{lemma}
  \label{long}
   Pour tout $z\in D_{\rm rig}$ et tout $j\geq 1$ on a
    $$(u^-)^j(z)=\frac{\nabla(\nabla-1)...(\nabla-j+1)(k-\nabla)(k+1-\nabla)...(k+j-1-\nabla)(z)}{t^j}.$$

  \end{lemma}

\begin{proof}
 Ceci d\'{e}coule par une r\'{e}currence imm\'{e}diate du th\'{e}or\`{e}me \ref{main1}, en utilisant la formule
 $$\nabla(k-\nabla)\left(\frac{x}{t^j}\right)=\frac{(k+j-\nabla)(\nabla-j)(x)}{t^j},$$
 qui se v\'{e}rifie par un calcul direct.
\end{proof}

  La conclusion du lemme \ref{long} (avec $j=k$) et du premier paragraphe est que $(u^{-})^k$ tue $\Pi_{c}^{P-\rm alg}$
  si et seulement si $\varphi^{-n}\left({\rm Res}_{\zp}(z)\right)\in X_n$ (voir le lemme \ref{diff} pour la d\'{e}finition de
   $X_n$) pour tout $z\in (\check{\Pi}^{\rm an})^*$ et tout $n\geq m(D)$. La partie a) du lemme \ref{diff} permet de conclure directement dans le cas de Rham.

 Supposons donc que $V$ n'est pas de de Rham et que $(u^{-})^k$ tue $\Pi_{c}^{P-\rm alg}$.
 Soit $z_1\in D^{\psi=1}$. La proposition V.2.1 de~\cite{Cbigone} (plut\^{o}t sa preuve) montre
 qu'il existe $z\in (\check{\Pi}^{\rm an})^*$ tel que ${\rm Res}_{\zp}(z)=z_1$. On a donc
 $\varphi^{-n}(z_1)\in X_n$. Mais $D$ a une base $z_1,z_2$ sur $\mathcal{E}$ form\'{e}e d'\'{e}l\'{e}ments de $D^{\psi=1}$ (voir~\cite[corollaire I.7.6]{CCthIwas}),
 et $z_1,z_2$ forment aussi une base de $D^{\dagger}$ sur $\mathcal{E}^{\dagger}$ (car $\mathcal{E}^{\dagger}$ est un corps),
 donc de $D_{\rm rig}$ sur $\mathcal{R}$. On en d\'{e}duit que pour tout\footnote{Cela peut demander d'augmenter $m(D)$.}
 $n\geq m(D)$, $z_1$ et $z_2$ forment une base de $D^{]0,r_n]}$ sur $\mathcal{E}^{]0,r_n]}$. Comme
 $\varphi^{-n}(z_1)$ et $\varphi^{-n}(z_2)\in X_n$ et comme $X_n$ est un $L_n[[t]]$-module, on obtient
 $\varphi^{-n}(D^{]0,r_n]})\subset X_n$ et donc $D_{{\rm dif},n}^+\subset X_n$. Mais cela contredit le point b) du lemme
 \ref{diff}. Cela permet de conclure.

  \end{proof}

   La preuve du th\'{e}or\`{e}me suivant (c'est le th\'{e}or\`{e}me \ref{main2} de l'introduction) est maintenant une formalit\'{e}.

   \begin{theorem}
    $V$ est de de Rham \`{a} poids de Hodge-Tate distincts si et seulement si
    $\Pi$ a des vecteurs localement alg\'{e}briques non nuls.
   \end{theorem}

   \begin{proof}
   Quitte \`{a} faire une torsion par un caract\`{e}re localement alg\'{e}brique,
   on peut supposer qu'un des poids de Hodge-Tate de $V$ est
   $0$. La proposition \ref{HTdif} permet de supposer que l'autre poids de Hodge-Tate est
   $k$, un entier strictement positif.

     Supposons que $V$ n'est pas de de Rham, mais que $\Pi^{{\rm alg}}\ne 0$.
D'apr\`{e}s\footnote{C'est une cons\'{e}quence de l'irr\'{e}ductibilit\'{e} sous l'action du sous-groupe de Borel de l'espace des fonctions localement constantes \`{a} support compact dans
$\qpet$, voir le lemme 2.9.1 de \cite{JL}.} \cite{Cbigone}, cor. VI.5.9, on a $\Pi_c^{P-{\rm alg}}\subset \Pi^{{\rm alg}}$
   et d'apr\`{e}s la proposition \ref{HTdif}, on sait que
   $\Pi^{{\rm alg}}={\rm Sym}^{k-1}\otimes \Pi^{\rm lc}$
   pour une repr\'{e}sentation lisse $\Pi^{\rm lc}$.
   En particulier, $(u^-)^k$ tue $\Pi_c^{P-{\rm alg}}$, contradiction avec le th\'{e}or\`{e}me
   \ref{crucialalg}.

   Supposons que $V$ est de de Rham. Fixons $n>m(D)$ et soit $e_1,e_2$ comme dans la proposition \ref{basedif}
 (c'est une base de $D_{\rm dR}(V)$). D'apr\`{e}s la proposition \ref{picpgros}, il existe
 $v\in \Pi_c^{P-{\rm alg}}$ tel que $$\phi_v(p^i)=1_{i=0}t^{k-1}e_2\pmod {D_{{\rm dif},n}^+}.$$
$\phi_v$ est alors \`{a} support dans $\zpet$ (prop \ref{Kir}) et, si on pose $u=\left(\begin{smallmatrix} 1 & 1 \\0 & 1\end{smallmatrix}\right)$, on a (en utilisant l'\'{e}galit\'{e} $[(1+T)^x]=(1+T)^x=e^{tx}$ pour $x\in\zp$)
  $$\phi_{(u-1)v}(x)=([(1+T)^{x}]-1)\phi_v(x)=1_{x\in \zpet}(e^{tx}-1)\phi_v(x)
  =0\pmod {D_{{\rm dif},n}^+},$$ donc $uv=v$, i.e. $v$ est invariant par $\left(\begin{smallmatrix} 1 & \zp \\0 & 1\end{smallmatrix}\right)$.
Si $a_n^+=\left(\begin{smallmatrix} 1+p^n & 0 \\0 & 1\end{smallmatrix}\right)$, alors $a_n^+v=(1+p^n)^{k-1}v$, car pour tout $x\in \qpet$ on a
  $$\phi_{a_n^+v-(1+p^n)^{k-1}v}(x)=\phi_v((1+p^n)x)-(1+p^n)^{k-1}\phi_v(x)=$$
 $$ 1_{x\in \zpet}\sigma_x(\sigma_{1+p^n}-(1+p^n)^{k-1})(t^{k-1}e_2)=0\pmod {D_{{\rm dif},n}^+}.$$
  Enfin, notons que $v\ne 0$, car
 $\phi_v(1)=t^{k-1}e_2\ne 0\pmod {D_{{\rm dif},n}^+}$.

 D'apr\`{e}s la proposition \ref{deep}, on a $v\in \Pi^{\rm an}$ et le paragraphe pr\'{e}c\'{e}dent montre que $u^+(v)=0$ et $h(v)=(k-1)v$. On va montrer que $v$ est localement alg\'{e}brique. D'apr\`{e}s la proposition \ref{sorites}, il suffit de montrer la $\mathfrak{gl_2}$-\'{e}quivariance du morphisme $L$-lin\'{e}aire  $f:{\rm Sym}^{k-1}(L^2)\to \Pi^{\rm an}$, d\'{e}fini par $$f(e_1^{k-j-1}e_2^j)=\frac{(u^-)^j v}{(k-1)(k-2)...(k-j)}$$
  pour $0\leq j\leq k-1$.

 Si $x_j=e_1^{k-j-1}e_2^j$, un calcul imm\'{e}diat montre que $$h(x_j)=(k-2j-1)x_j,\quad
 u^+(x_j)=jx_{j-1},\quad u^-(x_j)=(k-1-j)x_{j+1},$$ avec, par convention, $x_{-1}=x_{k}=0$. Soit
 $y_j=\frac{(u^-)^j v}{(k-1)(k-2)...(k-j)}$. On veut montrer que les $y_j$ satisfont les m\^{e}mes
 relations. Mais les relations $$[u^+,u^-]=h, \quad [h,u^+]=2u^+, \quad [h,u^-]=-2u^-$$ et une r\'{e}currence
 imm\'{e}diate montrent que, dans $U(\mathfrak{gl_2})$,
 $$(u^-)^jh=(h+2j)(u^-)^j,\quad u^+(u^-)^j=(u^-)^ju^++j(u^-)^{j-1}(h-j+1).$$
 La premi\`{e}re et le fait que $hv=(k-1)v$ entra\^{i}nent
 $h(y_j)=(k-1-2j)y_j$. La seconde et la relation $u^+v=0$
 donnent $u^+(y_j)=jy_{j-1}$. Enfin, il est clair que
 $u^-(y_j)=(k-1-j)y_{j+1}$ \textbf{si} $j<k-1$ et tout le point
 est de le v\'{e}rifier pour $j=k-1$, i.e. que $(u^-)^kv=0$.
Mais cela d\'{e}coule du th\'{e}or\`{e}me \ref{crucialalg} et finit la preuve du th\'{e}or\`{e}me.

\end{proof}

\end{document}